\documentclass[a4paper]{amsart}
\usepackage{amsthm, amssymb}
\usepackage{amsmath}
\usepackage{color}
\usepackage{comment}
\usepackage{tikz}
\usepackage{epic,eepic}
\usepackage{pict2e}
\usepackage[T1]{fontenc}
\usepackage[ascii]{inputenc} 
\usepackage{graphicx}
\usepackage{tikz-cd}
\usepackage{appendix}
\usepackage{amscd}
\usepackage[all]{xy}
\usepackage{accents}
\usetikzlibrary{shapes,snakes}
\SelectTips{cm}{}
\theoremstyle{plain}

\allowdisplaybreaks[1]

\newtheorem{thm}{Theorem}

\newtheorem{lem}[thm]{Lemma}

\newtheorem{claim}[thm]{Claim}
\newtheorem*{acknowledgments}{Acknowledgments}
\newtheorem*{organization}{Organization of this paper}\def\co{\colon\thinspace}

\begin{document}

\makeatletter

\newcommand*\@KP@Large@frame[2]{%
    \setlength\unitlength{\fontdimen 22 #1\tw@}%
    \vrule \@width\z@ \@height 4\unitlength \@depth\tw@\unitlength
    \begin{picture}(6,2)(-3,-1)%
        \def\@KP@Radius     {3}%
        \def\@KP@Hole@radius{.5}
        \def\@KP@Diameter   {6}%
        #2%
    \end{picture}%
}
\newcommand*\@KP@Small@frame[2]{%
    \setlength\unitlength{\fontdimen 22 #1\tw@}%
    \vrule \@width\z@ \@height \thr@@\unitlength \@depth\@ne\unitlength
    \begin{picture}(4,2)(-2,-1)%
        \def\@KP@Radius     {2}%
        \def\@KP@Hole@radius{.5}
        \def\@KP@Diameter   {4}%
        #2%
    \end{picture}%
}

\newcommand*\@KP@Radius     {}
\newcommand*\@KP@Hole@radius{}
\newcommand*\@KP@Diameter   {}
%
\newcommand*\@KP@Shape@A{%
    \put(0,0){\circle{\@KP@Diameter}}%
}
\newcommand*\@KP@Shape@B{%
    \Line(-\@KP@Radius,\@KP@Radius )(\@KP@Radius,-\@KP@Radius)%
    \Line(-\@KP@Radius,-\@KP@Radius)(-\@KP@Hole@radius,-\@KP@Hole@radius)%
    \Line(\@KP@Radius ,\@KP@Radius )(\@KP@Hole@radius ,\@KP@Hole@radius )%
}
\newcommand*\@KP@Shape@C{%
    \cbezier(-\@KP@Radius,\@KP@Radius )(0,0)(0,0)(\@KP@Radius,\@KP@Radius )%
    \cbezier(-\@KP@Radius,-\@KP@Radius)(0,0)(0,0)(\@KP@Radius,-\@KP@Radius)%
}
\newcommand*\@KP@Shape@D{%
    \cbezier(-\@KP@Radius,-\@KP@Radius)(0,0)(0,0)(-\@KP@Radius,\@KP@Radius)%
    \cbezier(\@KP@Radius ,-\@KP@Radius)(0,0)(0,0)(\@KP@Radius ,\@KP@Radius)%
}
\newcommand*\@KP@Shape@E{%
    \Line(\@KP@Radius,\@KP@Radius)(-\@KP@Radius,-\@KP@Radius)%
    \Line(-\@KP@Radius,\@KP@Radius )(-\@KP@Hole@radius,\@KP@Hole@radius)%
    \Line(\@KP@Radius ,-\@KP@Radius )(\@KP@Hole@radius ,-\@KP@Hole@radius )%
}

\newcommand*\@KP@Atomic@mathpalette[1]{%
    \mathinner{
        \mathchoice{%
            \linethickness{.6\p@}
            \@KP@Large@frame \textfont {#1}%
        }{%
            \linethickness{.4\p@}
            \@KP@Small@frame \textfont {#1}%
        }{%
            \linethickness{.3\p@}
            \@KP@Small@frame \scriptfont {#1}%
        }{%
            \linethickness{.2\p@}
            \@KP@Small@frame \scriptscriptfont {#1}%
        }%
    }%
}

\newcommand*\KPA{\@KP@Atomic@mathpalette \@KP@Shape@A}
\newcommand*\KPB{\@KP@Atomic@mathpalette \@KP@Shape@B}
\newcommand*\KPC{\@KP@Atomic@mathpalette \@KP@Shape@C}
\newcommand*\KPD{\@KP@Atomic@mathpalette \@KP@Shape@D}
\newcommand*\KPE{\@KP@Atomic@mathpalette \@KP@Shape@E}

\renewcommand{\appendixname}{Appendix }
\makeatother

\title{The skein algebra of  the Borromean rings complement}

\author{Go Miura}
\address{Brains Technology, Inc., Tokyo, 108-0074, Japan.}
	\email{miura.go@brains-tech.co.jp}
	\author{Sakie Suzuki}
\address{Department of Mathematical and Computing Science, School of Computing, Tokyo Institute of Technology, Tokyo, 152-8550, Japan.}
	\email{sakie@c.titech.ac.jp }
\date{}
\maketitle

\begin{abstract}
The skein algebra of an oriented $3$-manifold is a classical limit of the Kauffman bracket skein module and gives the coordinate ring of the $SL_2(\Bbb{C})$-character variety. In this paper we determine the quotient of a polynomial ring which is isomorphic to the skein algebra of a group with three generators and two relators. As an application, we give an explicit formula for the skein algebra of the Borromean rings complement in $S^3$.
\end{abstract}

\tableofcontents

\section{Introduction}

\subsection{Classical limit of Kauffman bracket skein module and skein algebra}

Let $M$ be an an oriented $3$-manifold and $R$  a commutative ring with unit $1$.
Przytycki \cite{przytycki1, przytycki2} and Turaev \cite{turaev1, turaev2}  introduced  the Kauffman bracket skein module (KBSM) $S(M; R, t)$
as a generalization of the Kauffman bracket in $S^3$. The KBSM of $M$ is defined by the Kauffman bracket skein relation $\KPB -t\KPC-t^{-1}\KPD$ and the weight relation $\KPA + (t^{2} + t^{-2})\emptyset$ dividing the free $R$-module spanned by framed links in $M$. 
For link complements in $S^3$, the KBSM of two-bridge links  are calculated \cite{bullock1, BL, le, letran}.  For other links and general parameter $t$, the KBSM is less researched because its algebraic structure is too complicated.

%
%
%
%

The classical case $t=-1$ of the KBSM is relatively simple to be calculated and is important; in this case the framing becomes irrelevant and 
 $S(M; R, -1)$ admits a commutative algebra structure with the multiplication as the disjoint union of links. Bullock \cite{bullock} showed that the  algebra 
$S(M;\Bbb{C},-1)$
 modulo its nilradical  is isomorphic to the coordinate ring of the $SL_2(\Bbb{C})$-character variety of $M$.  
From this point of view the KBSM can be seen as a deformation of $SL_2(\Bbb{C})$-character variety, and plays an important role in the study of the AJ conjecture \cite{FGL, le, letran2}, which relates the recurrence polynomial of the colored Jones polynomial and the A-polynomial.

In \cite{przytyckisikora2, przytyckisikora}, Przytycki and Sikora introduced the skein algebra $S(G; R)$ of a group $G$, which is a generalization of $S(M; R, -1)$ in the sense that  $S(M; R, -1)$ can be recovered when $G=\pi_1 (M)$.  
They also showed that the skein algebra $S(G;\Bbb{C})$ modulo its nilradical is isomorphic to the coordinate ring of the $SL_2(\Bbb{C})$-character variety of $G$. 

For link complements in $S^3$, Tran \cite{tran1, tran2} calculated the skein algebras of pretzel links. Note that the fundamental groups of the complements of pretzel links and  two-bridge links  are generated by two elements.
In this paper we determine the skein algebra of a group with three generators and two relators, and give an explicit formula for the skein algebra, i.e., the  KBSM with $t=-1$, of the Borromean rings complement in $S^3$.	

\subsection{Results}
Let $G$ be a group, $R$ a commutative ring with unit $1$ and $SR[G] = \bigoplus_m SR[G]_m$ the symmetric algebra over the group ring $R[G]$. Consider the ideal $I$ of $SR[G]$ generated by $e-2$ and $g\otimes h - gh - gh^{-1}$ for $g, h\in G$. Here we denote by $e\in SR[G]_1$ the identity in $G$, and by $2\in SR[G]_0$ the sum of units $1+1 \in R$. Then the skein algebra $S(G;R)$  is defined  as the quotient algebra $SR[G]/I$. In this paper we denote by $[g]$ the equivalence class of $g\in G$ in $SR[G]/I = S(G;R)$.

Horowitz \cite{horowitz}, Culler and Shalen \cite{culler-shalen}  showed that  if $G$ is generated by $g_1, g_2, \ldots, g_n$, then each $SL_2(\Bbb{C})$-character $\chi$ is determined by the values $\chi(g_{i_1})\cdots \chi(g_{i_k})$,  $1\leq k \leq n, 1\leq i_1< i_2 < \cdots < i_k \leq n$. This result implies that 
the skein  algebra $S(G;\Bbb{C})$ is generated by $\{[g_{i_1}\cdots g_{i_k}]\ |\ 1\leq k\leq n, 1\leq i_1<\cdots<i_k\leq n\}$. The main interest of the present paper is the relations among these generators.

Let $F_n$ be the free group with generators $g_1,\ldots,g_n$. As for the free groups $F_1$ and $F_2$, there are no relations among the above generators of $S(F_1;\Bbb{C})$ and $S(F_2;\Bbb{C})$, and therefore $S(F_1;\Bbb{C})$ and $S(F_2;\Bbb{C})$ are nothing but the polynomial rings $\Bbb{C}[[g_1]]$ and $\Bbb{C}[[g_1],[g_2],[g_1g_2]]$, respectively. However, as for $F_n$ with $n$ larger than 2, the skein algebra is no longer a polynomial ring of the variables $\{[g_{i_1}\cdots g_{i_k}]\ |\ 1\leq k\leq n, 1\leq i_1<\cdots<i_k\leq n\}$, but is the quotient of the polynomial ring by a certain ideal. In order to describe this ideal we consider the algebra homomorphism from the polynomial ring $\Bbb{C}[x_1,\ldots,x_{1\cdots n}]$ with $2^n-1$ variables $x_{{i_1}\cdots{i_k}}\ (1\leq k \leq n, 1\leq i_1< i_2 < \cdots < i_k \leq n)$ to the skein algebra $S(F_n;\Bbb{C})$ as follows:
\begin{align*}
\Phi \co \Bbb{C}[x_1,\ldots,x_{1\cdots n}] &\rightarrow S(F_n;\Bbb{C})\\
x_{{i_1}\cdots{i_k}} &\mapsto [g_{i_1}\cdots g_{i_k}].
\end{align*}  
Since $\Phi$ is surjective, $S(F_n;\Bbb{C})$ is isomorphic to $\Bbb{C}[x_1,\ldots,x_{1\cdots n}]/\ker\Phi$. For $n=3$, by the argument in \cite{horowitz}, we have $\ker\Phi=\langle K\rangle$, where
\begin{equation*}
\begin{split}
K &:= x_{123}^2 - (x_{12}x_3 + x_{13}x_2 + x_{23}x_1 - x_1x_2x_3)x_{123} + x_1^2+x_2^2 + x_3^2 + x_{12}^2 +x_{23}^2 + x_{13}^2\\
 &\quad - x_1x_2x_{12} - x_1x_3x_{13}- x_2x_3x_{23} + x_{12}x_{13}x_{23} -4,
\end{split}
\end{equation*}
i.e., we have 
$$S(F_3; \Bbb{C})=\Bbb{C}[x_1,\ldots, x_{123}]/ \langle K\rangle.$$ 

Let $G$ be a group with generators $g_1,\ldots, g_n$ and set $\iota\co F_n\ni g\mapsto \bar g\in G$ the natural group homomorphism such that $\iota(g_i) = g_i$ for $1\leq i\leq n$. We extend $\iota$ to the algebra homomorphism $\iota \co S(F_n;\Bbb{C})\rightarrow S(G;\Bbb{C})$ and define $\Phi_G:=\iota \circ \Phi$. Then we have $\Bbb{C}[x_1,\ldots,x_{1\cdots n}]/\ker\Phi_G\cong S(G;\Bbb{C})$. 

For $n=1$ or $2$, it is known that
$$\ker\Phi_G=\langle P_u - P_v\ |\ u,v\in F_n\mathrm{\ s.t.\ }\bar u=\bar v \rangle,$$
where $P_g$ is any polynomial such that $\Phi(P_g)=[g]$.

We determine $\ker\Phi_G$ for $n=3$ as follows.

\begin{thm}\label{Theorem:1}
Let $G$ be a group with three generators. Then we have
$$\ker\Phi_G=I_G :=\langle K, P_u - P_v\ |\ u,v\in F_3\mbox{\ s.t.\ } \bar u=\bar v\rangle,$$
where we fix a polynomial $P_x\in\Phi^{-1}([x])$ for $x\in F_3$. Therefore we have
$$S(G; \Bbb{C})\cong\Bbb{C}[x_1, x_2, x_3, x_{12}, x_{13}, x_{23}, x_{123}] / I_G.$$
\end{thm}

Note that $I_G$ is generated by a finite set of polynomials, since the polynomial ring is Noetherian. The following theorem gives a finite set of generators of $I_G$ when $G$ has two relators.

\begin{thm}\label{Theorem:2}
Let $G$ be a group defined by three generators $g_1,g_2,g_3$ and two relators $\alpha=\beta$ and $\gamma=\delta$. Then $I_G$ is generated by 
$$\{K, P_{\alpha g}-P_{\beta g}, P_{\gamma g}-P_{\delta g}\ |\ g=g_1^{i_1}g_2^{i_2}g_3^{i_3}, 0\leq i_1, i_2, i_3\leq1\}.$$ 
\end{thm}

By the Wirtinger presentation of the Borromean rings $B$, we have
$$\pi_1M_B =\langle  g_1, g_2, g_3\ |\  \alpha = \beta,
\gamma = \delta\rangle$$
where $\alpha = g_3g_2^{-1}g_1g_2g_1^{-1}, \beta=g_2^{-1}g_1g_2 g_1^{-1}g_3, \gamma=g_2g_1^{-1}g_3g_1g_3^{-1}$ and $\delta=g_1^{-1}g_3g_1g_3^{-1}g_2$ (Lemma \ref{pi1ofBorromean}). Using Theorem \ref{Theorem:2} we give explicit generators of $I_{\pi_1M_B}$ as follows. 

\begin{thm}\label{Theorem:3}
For the Borromean rings complement $M_B$, let $I_{\pi_1M_B}$ be the ideal generated by the polynomials below. Then we have $S(\pi_1M_B;\Bbb{C}) \cong \Bbb{C}[x_1,\ldots,x_{123}]/I_{\pi_1M_B}$. 
\begin{align*}
K &= x_{123}^2 - (x_1x_{23} + x_2x_{13} + x_3x_{12} - x_1x_2x_3)x_{123} + x_1^2+x_2^2 + x_3^2 \\
&\quad + x_{12}^2 + x_{13}^2 + x_{23}^2 - x_1x_2x_{12} - x_1x_3x_{13}- x_2x_3x_{23} + x_{12}x_{13}x_{23} -4,\\
Q_{\alpha g_1,\beta g_1} &= -2 x_{12} x_{23} + 2 x_2 x_{123} - {x_1}^2 x_2 x_{123} + x_1 x_2 x_{12} x_{13} + x_1 x_2 x_{23}\\ 
&\quad + x_1 x_{12} x_{123} - {x_{12}}^2 x_{13} - {x_2}^2 x_{13},\\
Q_{\gamma g_1,\delta g_1} &= 2x_{13}x_{23} - 2x_3x_{123} + x_1^2x_3x_{123} - x_1x_3x_{12}x_{13} - x_1x_3x_{23}\\
&\quad - x_1x_{13}x_{123} + x_{12}x_{13}^2 +x_3^2x_{12},\\
Q_{\alpha g_2,\beta g_2} &= -x_1x_2x_{12}x_{23} + {x_{12}}^2x_{23} + {x_1}^2x_{23} + x_1x_2^2x_{123} - x_2x_{12}x_{123}\\ 
&\quad - 2x_1x_{123} - x_{1}x_2x_{13} + 2x_{12}x_{13},\\
Q_{\gamma g_3,\delta g_3} &= x_1x_3x_{13}x_{23} - {x_{13}}^2x_{23} - {x_1}^2x_{23} - x_1x_3^2x_{123} + x_3x_{13}x_{123}\\ 
&\quad + 2x_1x_{123} + x_{1}x_3x_{12} -2x_{12}x_{13},\\
Q_{\alpha g_1g_2,\beta g_1g_2} &= -x_1^2x_{123} + x_2^2x_{123} + x_1x_{12}x_{13} - x_2x_{12}x_{23} - 2x_2x_{13} + 2x_1x_{23},\\
Q_{\gamma g_1g_2,\delta g_1g_2} &= x_1^3 +  x_1x_3^2 +x_1x_{13}^2 -x_1^2x_3x_{13}  - 4x_1 +x_1^2x_2x_3x_{123} - x_1x_2x_{13}x_{123}\\
&\quad  -x_1x_2x_3x_{23} -x_1x_3x_{12}x_{123} +x_{12}x_{13}x_{123} - x_1^2x_2x_{12} + x_1x_{12}^2\\
&\quad + x_3x_{12}x_{23} - x_2x_3x_{123}+ x_2x_{13}x_{23} + x_1x_2^2,\\
Q_{\alpha g_1g_3,\beta g_1g_3} &= -4x_1 + x_1^3+ x_1x_{12}^2 + x_{12}x_{13}x_{123} + x_1x_{13}^2 - x_1^2x_2x_{12} - x_1x_2x_{13}x_{123}\\
&\quad + x_1x_2^2 + x_2x_{13}x_{23} - x_1^2x_3x_{13} - x_3x_{12}^2x_{13} + x_2x_3x_{123}\\
&\quad + x_1x_2x_3x_{12}x_{13} - x_2^2x_3x_{13} - x_3x_{12}x_{23} + x_1x_3^2,\\
Q_{\gamma g_1g_3,\delta g_1g_3} &= x_1^2 x_{123} - x_3^2x_{123} - x_1^3 x_{23} + x_3^3x_{12} - x_1 x_{12} x_{13} + x_3x_{13} x_{23}\\
&\quad + 2 x_1 x_{23} - 2 x_3x_{12} - x_1 x_{13}^2 x_{23} + x_3x_{12} x_{13}^2 + x_1^2 x_3 x_{12} - x_1x_3^2x_{23}\\
&\quad + x_1^2x_3 x_{13} x_{23} - x_1 x_3^2x_{12} x_{13},\\
Q_{\alpha g_2g_3,\beta g_2g_3} &= x_2^3 + x_2x_3^2 + x_2x_{23}^2 - x_2^2x_3x_{23} - 4x_2 + x_1x_2^2x_3x_{123} - x_1x_2x_{23}x_{123}\\
&\quad - x_1x_2x_3x_{13} - x_2x_3x_{12}x_{123} + x_{12}x_{23}x_{123} - x_1x_2^2x_{12} + x_2x_{12}^2\\
&\quad + x_3x_{12}x_{13} - x_1x_3x_{123} + x_1x_{13}x_{23} + x_1^2x_2,\\
Q_{\gamma g_2g_3,\delta g_2g_3} &= -x_3^3 -x_1^2x_3 -x_3x_{13}^2 +x_1x_3^2x_{13}  +4x_3 -x_1x_2x_3^2x_{123} + x_2x_3x_{13}x_{123}\\
&\quad  +x_1x_2x_3x_{12} +x_1x_3x_{23}x_{123} -x_{13}x_{23}x_{123} +x_2x_3^2x_{23} -x_3x_{23}^2\\
&\quad -x_1x_{12}x_{23} +x_1x_2x_{123} -x_2x_{12}x_{13} -x_2^2x_3,\\
Q_{\alpha g_1g_2g_3,\beta g_1g_2g_3} &= -x_2x_3x_{12}x_{23} + x_3^2x_{12} +x_1x_3x_{23} -x_2x_{3}x_{13} - 4x_{12} +x_2^2x_{12}\\ 
&\quad +x_{12}^3 +x_1^2x_{12} -x_1x_2x_{12}^2 -x_1x_2x_{123}^2 +x_{12}x_{123}^2\\
&\quad +x_1x_2x_3x_{12}x_{123} -x_3x_{12}^2x_{123} +x_2x_{23}x_{123} -x_1^2x_3x_{123} +x_1x_{13}x_{123},\\
Q_{\gamma g_1g_2g_3,\delta g_1g_2g_3} &= (-x_1x_{23} +x_3x_{12})x_{13}x_{123} +(x_1x_{12} -x_3x_{23})x_{123}\\
&\quad +(x_3^2-x_1^2)x_{12}x_{23} +(x_{23}^2-x_{12}^2)x_{13} +(x_1^2x_3x_{23} -x_1x_3^2x_{12})x_{123}\\
&\quad +(x_{12}^2-x_{23}^2)x_{1}x_3 +(x_1x_{23}-x_3x_{12})x_2.
\end{align*}
\end{thm}

We can find some symmetries in the polynomials above. For example, if we multiply $Q_{\alpha g_1,\beta g_1}$ by $-1$ and exchange each "2" and "3"  in the indices of the variables except for $x_{23}$ and $x_{123}$, then we obtain $Q_{\gamma g_1,\delta g_1}$. This kind of symmetry can be observed between $Q_{\alpha g_1,\beta g_1}$ and $Q_{\gamma g_1,\delta g_1}$, $Q_{\alpha g_2,\beta g_2}$ and $Q_{\gamma g_3, \delta g_3}$ and among $Q_{\gamma g_1g_2,\delta g_1g_2}$, $Q_{\alpha g_2g_3,\beta g_2g_3}$ and $Q_{\gamma g_2g_3, \delta g_2g_3}$. Each $Q_{\alpha g_1g_2, \beta g_1g_2}, Q_{\gamma g_1g_3, \delta g_1g_3}$ and $Q_{\gamma g_1g_2g_3, \delta g_1g_2g_3}$ has a symmetry by themselves.

\begin{organization}
\rm{The rest of the paper is organized as follows. In Section \ref{FS} we study the skein algebra of groups with three generators and two relators, and prove Theorem \ref{Theorem:1} and Theorem \ref{Theorem:2}. In Section \ref{SS} we consider the skein algebra of  the complement $M_B$ of the Borromean rings $B$. In Section \ref{SS_1} using  Wirtinger's method we give a presentation of  $\pi_1M_B$ with three generators and two relators,  and in Section \ref{SS_2} we prove Theorem \ref{Theorem:3}. }

\end{organization}

\begin{acknowledgments}
\rm{We would like to thank K.Habiro, K.Okazaki,  Y. Ota, and A. Tran for valuable discussions. This work is partially supported by JSPS KAKENHI Grant Number JP19K14523.}
\end{acknowledgments}

\section{Skein algebra of groups with three generators and two relators}\label{FS}
In this section we prove Theorem \ref{Theorem:1} and Theorem \ref{Theorem:2}.

\subsection{Proof of Theorem \ref{Theorem:1}.}\label{FS_3}
We prove Theorem \ref{Theorem:1}.  
Let  $G$  be a group with three generators $g_1,g_2,g_3$. Recall from the introduction the natural group homomorphism $\iota : F_3\ni g\mapsto \bar g\in G$ 
such that $\iota(g_i)=g_i$ for $i=1,2,3$.
We use the following lemma.
\begin{lem}\label{prop:15}
$S(G; \Bbb{C})$ is isomorphic to $S(F_3; \Bbb{C}) / \langle\ [u] - [v]\ |\ u, v\in F_3\mbox{\ s.t.\ }\bar{u}=\bar{v}\ \rangle $. 
\end{lem}

\begin{proof} 
Let us consider the following two ideals of the symmetric algebra $S\Bbb{C}[F_3]$:
\begin{align*}
I &:= \langle e-2, g\otimes h - gh - gh^{-1}\ |\ g,h\in F_3\rangle, \\
\bar{I} &:= \langle e-2, g\otimes h - gh - gh^{-1}, u - v\ |\ g,h\in F_3,\ u,v\in F_3\mbox{\ s.t.\ }\bar u=\bar v\rangle.
\end{align*}
If we consider the canonical projection 
$$\pi \co S\Bbb{C}[F_3] \rightarrow S\Bbb{C}[F_3]/I = S(F_3;\Bbb{C}),$$ 
then we have
\begin{align*}
\bar I/I &= \pi(\bar I) = \langle \pi(u)-\pi(v)\ |\ u,v\in F_3\mbox{\ s.t.\ }\bar u=\bar v \rangle\\
&= \langle [u] - [v]\ |\ u, v\in F_3\mbox{\ s.t.\ }\bar{u}=\bar{v} \rangle.
\end{align*}
Thus by the third isomorphism theorem, we have
$$S(F_3;\Bbb{C})/\langle [u] - [v]\ |\ u, v\in F_3\mbox{\ s.t.\ }\bar{u}=\bar{v} \rangle =\left(S\Bbb{C}[F_3] / I \right) / \left( \bar{I}/I \right)\cong
S\Bbb{C}[F_3] / \bar{I}.$$
Hence it suffices to show that $S\Bbb{C}[F_3] / \bar{I} \cong S(G ; \Bbb{C})$. 

We extend the group homomorphism $\iota \co F_3\rightarrow G$ to the algebra homomorphism $\iota \co S\Bbb{C}[F_3] \rightarrow S\Bbb{C}[G]$. Then  $\iota$ induces the algebra isomorphism $\hat\iota \co S\Bbb{C}[F_3]/\ker\iota \rightarrow S\Bbb{C}[G]$, where the kernel $\ker\iota $ is the ideal generated by $u-v$, $u, v\in F_3$, such that $\bar{u}=\bar{v}$.
Then again by the third isomorphism theorem, we have
$$S\Bbb{C}[F_3]/\bar{I} \cong (S\Bbb{C}[F_3]/\ker\iota)/(\bar{I}/\ker\iota) \cong S\Bbb{C}[G]/ \hat\iota(\bar{I}/\ker\iota)\cong S(G;\Bbb{C}),$$
where the last identity is given by
$$\hat\iota(\bar I / \ker\iota) = \langle \bar e-2, \bar g\otimes\bar h - \bar g\bar h - \bar g\bar h^{-1}\ |\ g, h\in F_3 \rangle,$$
which is the defining relation of the skein algebra $S(G;\Bbb{C})$.
\end{proof}

\begin{proof}[Proof of Theorem \ref{Theorem:1}]
Recall from the introduction that the map $\Phi\co \Bbb{C}[x_1,\ldots,x_{123}] \to S(F_3;\Bbb{C})$ has the kernel generated by the polynomial $K$, i,e., we have the  isomorphism 
\begin{align*}
\hat{\Phi} \co \Bbb{C}[x_1, \ldots, x_{123}] / \langle K \rangle \rightarrow S(F_3; \Bbb{C}).
\end{align*}
Then by the third isomorphism theorem, we have
\begin{align*}
\Bbb{C}[x_1, \ldots, x_{123}] / I_G &\cong \left(\Bbb{C}[x_1, \ldots, x_{123}] / \langle K \rangle \right) / \left(I_G / \langle K \rangle \right),
\\
&\cong S(F_3; \Bbb{C}) / \langle\ [u] - [v]\ |\ u, v\in F_3\mbox{\ s.t.\ }\bar{u}=\bar{v}\ \rangle.
\end{align*}
By Lemma \ref{prop:15} the right hand side is isomorphic to $S(G ; \Bbb{C})$, which implies the assertion.
\end{proof}
\subsection{Basic identities in skein algebra}
The following lemma is well-known (cf.  \cite{horowitz, tran1}).

\begin{lem}\label{lem:skeineq}
Let $G$ be a group and $a, b, c\in G$ arbitrary elements. In $S(G;R)$, we have the following identities.
\begin{align*}
[a]&=[a^{-1}]\tag{1}\\
[ab]&=[ba]\tag{2}\\
[b]\otimes[ac]&=[abc]+[ab^{-1}c]\tag{3}\\
\tag{4}[abc]&=[a]\otimes[bc]+[b]\otimes[ac]+[c]\otimes[ab]-[a]\otimes[b]\otimes[c]-[acb]\\
\end{align*}
\end{lem}

\begin{proof}We have
\begin{align*}
[a] &= 2 [a] - [a] = [e]\otimes[a] - [a] = ([e\cdot a] + [e\cdot a^{-1}]) -[a] =  [a^{-1}],\\
[ab] &= [a]\otimes[b] - [ab^{-1}] = [b]\otimes[a] - [{(ab^{-1})}^{-1}]\\
&= [b]\otimes[a] - [ba^{-1}] = [b]\otimes[a] - ([b]\otimes[a]- [ba]) = [ba],\\ 
[b]\otimes[ac] &= [ca]\otimes[b] = [cab] + [ca{b}^{-1}] = [abc] + [ab^{-1}c],\\
[abc] &= [a(bc)] = [a]\otimes[bc] - [ac^{-1}b^{-1}] = [a]\otimes[bc] - ([ac^{-1}]\otimes[b] - [ac^{-1}b])\\
&= [a]\otimes[bc] - (-[ac]\otimes[b] + [a]\otimes[c]\otimes[b]) - [acb] + [ab]\otimes[c]\\
&=  [a]\otimes[bc] +[b]\otimes[ac] + [c]\otimes[ab] - [a]\otimes[b]\otimes[c] - [acb].\\
\end{align*}
\end{proof}

\subsection{Proof of Theorem \ref{Theorem:2}}\label{FS_4}
We prove Theorem \ref{Theorem:2}. 
Let
$$G=\langle g_1, g_2, g_3\ |\ \alpha=\beta, \gamma=\delta\rangle,$$
and consider the ideal 
$$\bar I_G := \langle K, P_{\alpha g}- P_{\beta g},  P_{\gamma g}-P_{\delta g}\ |\ g=g_1^{i_1}g_2^{i_2}g_3^{i_3}, 0\leq i_1, i_2, i_3\leq1\rangle \subset \Bbb{C}[x_1,\ldots,x_{123}].$$
Note that Theorem \ref{Theorem:2} is equivalent to $I_G=\bar I_G$.

Since generators of $\bar I_G$ are all contained in $I_G$, it suffices to show that $\bar I_G\supset I_G$. Since $K\in\bar I_G$, it is enough to prove the following lemma.

\begin{lem}\label{genlem:2}
For $u,v\in F_3$ such that $\bar u=\bar v$ and for $P_u\in \Phi^{-1}([u]), P_v\in \Phi^{-1}([v])$, we have $P_u- P_v \in\bar I_G$.
\end{lem}

\begin{proof}
Take $u,v\in F_3$ such that $\bar u=\bar v$. Since $G$ is the quotient of the free group $F_3$ by its normal subgroup
 $\langle \ h\alpha\beta^{-1}h^{-1}, h'\gamma\delta^{-1}h'^{-1} \ |\ h, h' \in F_3 \ \rangle,$ we can write 
$$u = \left(\prod_{i=1}^m h_i a_i b_i^{-1}h_i^{-1}\right)v$$ for $h_i\in F_3$ and $(a_i, b_i)\in\{(\alpha, \beta), (\beta, \alpha), (\gamma, \delta), (\delta, \gamma)\}$, $i=1,\ldots, m$. 

We reduce the assertion to the following claim:
\begin{claim}
For $j=1,\ldots,m$, set
\begin{eqnarray*}
c_j:=\left\{ \begin{array}{ll}
b_j^{-1}h_j^{-1}\left(\prod_{i=j+1}^m h_ia_ib_i^{-1}h_i^{-1}\right) v h_j & (1\leq j\leq m-1), \\
b_m^{-1}h_m^{-1}vh_m& (j=m). \\
\end{array} \right.
\end{eqnarray*}
For $P_u, P_v, P_{a_1c_1}$ and $P_{b_mc_m}$ we have 
\begin{itemize}
\item[\rm{(a)}] $P_u - P_{a_1c_1} \in \bar I_G$\\
\item[\rm{(b)}] $P_{a_1c_1} - P_{b_mc_m}\in \bar I_G$\\
\item[\rm{(c)}] $P_{b_mc_m} - P_v \in \bar I_G$\\
\end{itemize}
\end{claim}
We first prove (a) and (c). 
Note that 
\begin{align*}
[u] &= \left[\left( \prod_{i=1}^{m} h_{i} a_i b_i^{-1} h_i^{-1} \right) v\right]\\
&= \left[a_1b_1^{-1}h_1^{-1}\left(\prod_{i=2}^{m} h_{i} a_i b_i^{-1} h_i^{-1}\right)vh_1\right]\\
&= [a_1c_1],\\
& \\
[v] &= [h_m^{-1}vh_m]\\
&= [b_mb_m^{-1}h_m^{-1}vh_m]\\
&= [b_mc_m].
\end{align*}
Thus we have
\begin{align*}
\Phi(P_u - P_{a_1c_1}) &= [u]- [a_1c_1] = 0,\\
\Phi(P_{b_mc_m} - P_v) &= [b_mc_m] - [v] =0,
\end{align*}
which imply
\begin{align*}
P_u - P_{a_1c_1}&\in\ker\Phi=\langle K\rangle\subset\bar I_G,
\\ P_{b_mc_m} - P_{v}&\in\ker\Phi=\langle K\rangle\subset\bar I_G.
\end{align*}\\

In order to prove (b), we use the following claim.
\begin{claim}
For $P_{a_jc_j}, P_{b_jc_j}$ and $P_{a_{j+1}c_{j+1}}$ we have
\begin{itemize}
\item[\rm{(b1)}] $P_{a_jc_j} - P_{b_jc_j} \in\bar I_G$\ for $1\leq j\leq m$.\\
\item[\rm{(b2)}] $P_{b_jc_j} - P_{a_{j+1}c_{j+1}}\in\bar I_G$\ for $1\leq j\leq m-1$.
\end{itemize}
\end{claim}
By combining (b1) and (b2) we obtain (b);
$$P_{a_1c_1} - P_{b_mc_m} = \left(\sum_{j=1}^{m} P_{a_jc_j} - P_{b_jc_j}\right) + \left(\sum_{j=1}^{m-1}P_{b_jc_j} - P_{a_{j+1}c_{j+1}}\right)\in\bar I_G.$$
As for (b1), the proof is given under more general condition in Lemma \ref{genelem:1} below. 

We prove (b2). When $1\leq j\leq m-2$, we have
\begin{align*}
[b_jc_j]&= \left[h_j^{-1}\left(\prod_{i=j+1}^m h_ia_ib_i^{-1}h_i^{-1}\right) v h_j\right]\\
&= \left[\left(\prod_{i=j+1}^m h_ia_ib_i^{-1}h_i^{-1}\right)v\right]\\
&=\left[h_{j+1}a_{j+1}b_{j+1}^{-1}h_{j+1}^{-1}\left(\prod_{i=j+2}^m h_ia_ib_i^{-1}h_i^{-1}\right) v\right]\\
&= \left[a_{j+1}b_{j+1}^{-1}h_{j+1}^{-1}\left(\prod_{i=j+2}^m h_ia_ib_i^{-1}h_i^{-1}\right) v h_{j+1}\right]\\
&= [a_{j+1}c_{j+1}]
\end{align*}
and when $j=m-1$, we have
\begin{align*}
[b_{m-1}c_{m-1}] &= [h_{m-1}^{-1}h_ma_mb_m^{-1}h_m^{-1}vh_{m-1}]\\
&= [h_ma_mb_m^{-1}h_m^{-1}v]\\
&= [a_mb_m^{-1}h_m^{-1}vh_m]\\
&= [a_mc_m].
\end{align*}
Thus for $1\leq j\leq m-1$ we have $[b_jc_j]=[a_{j+1}c_{j+1}]$ and therefore
$$P_{b_jc_j} - P_{a_{j+1}c_{j+1}}\in\ker\Phi=\langle K\rangle\subset\bar I_G.$$
\end{proof}

\begin{lem}\label{genelem:1}
For any $g\in F_3$ we have $P_{\alpha g} -P_{\beta g},  P_{\gamma g} -P_{\delta g} \in\bar I_G$.
\end{lem} 

\begin{proof} 
Let $P_x$ denote an arbitrary fixed element in $\Phi^{-1}([x])$. We show $P_{\alpha g} -P_{\beta g}\in\bar I_G$. One can prove $P_{\gamma g} -P_{\delta g} \in\bar I_G$ in a similar way. Note that if there exist certain polynomials $Q_{\alpha g}\in \Phi^{-1}([\alpha g])$ and $Q_{\beta g}\in\Phi^{-1}([\beta g])$ such that $Q_{\alpha g} - Q_{\beta g}\in\bar I_G$, then we have
\begin{align*} \label{pq}
P_{\alpha g} - P_{\beta g} =  \left(P_{\alpha g} -Q_{\alpha g}\right) + \left(Q_{\alpha g} - Q_{\beta g}\right) + \left(Q_{\beta g} - P_{\beta g}\right)\in\bar I_G,
\end{align*}
since $P_{\alpha g} -Q_{\alpha g}, Q_{\beta g} - P_{\beta g}  \in\ker\Phi=\langle K\rangle\subset\bar I_G.$
In what follows we construct such $Q_{\alpha g}$ and $Q_{\beta g}$ for each $g=g_{i_1}^{m_1} \cdots g_{i_r}^{m_r}$, where  $g_{i_j} \in \{g_1, g_2, g_3\}$ and $m_j \in \Bbb{Z} \setminus \{0\}$. We assume that $g_{i_j}\neq g_{i_{j+1}}$ for $1\leq j \leq r-1$.
 \\

\noindent\fbox{case 1. $r=0$}\\
Let $g = e$. Then we have $P_{\alpha g} - P_{\beta g}=P_\alpha - P_\beta\in\bar I_G$.\\

\noindent\fbox{case 2. $r=1$}\\
Let $g = g_{i_1}^{m_1}$ for $m_1 \in \Bbb{Z} \setminus \{0\}$.\\

\begin{itemize}
\item subcase 2-1. $m_1 = 1$\\
In this case by the definition of $\bar I_G$ we have $P_{\alpha g_{i_1}} - P_{\beta g_{i_1}}\in\bar I_G$.\\
	
\item subcase 2-2. $m_1>1$\\
We use induction on $m_1$. Set
\begin{align*}
Q_{\alpha g_{i_1}^{m_1}} &:= P_{\alpha g_{i_1}^{m_1 - 1}}x_{i_1} - P_{\alpha g_{i_1}^{m_1 - 2}},\\
Q_{\beta g_{i_1}^{m_1}} &:= P_{\beta g_{i_1}^{m_1 - 1}} x_{i_1} - P_{\beta g_{i_1}^{m_1 - 2}},
\end{align*}
so that we have $\Phi(Q_{\alpha g_{i_1}^{m_1}}) = [\alpha g_{i_1}^{m_1}]$ and $\Phi(Q_{\beta g_{i_1}^{m_1}}) = [\beta g_{i_1}^{m_1}]$, which follows from the defining relation $[gh]=[g]\otimes [h]  - [gh^{-1}]$ in the skein algebra.
Then we have
\begin{align*}
Q_{\alpha g_{i_1}^{m_1}} - Q_{\beta g_{i_1}^{m_1}} &= ( P_{\alpha g_{i_1}^{m_1 - 1}} - P_{\beta g_{i_1}^{m_1 - 1}}) x_{i_1} - ( P_{\alpha g_{i_1}^{m_1 - 2}} - P_{\beta g_{i_1}^{m_1 - 2}} ) \in \bar I_G,
\end{align*}
where the assertion for $m_1=2$ follows  from the case 1 and the subcase 2-1, that
for $m_1=3$  follows  from the  subcase 2-1 and the case $m_1=2$, and that
for $m_1\geq 4$   follows from the induction on $m_1$.
\\

\item suncase 2-3. $m_1 < 0$\\
We use induction on $|m_1|$. Set
\begin{align*}
Q_{\alpha g_{i_1}^{m_1}} &:= P_{\alpha g_{i_1}^{m_1 + 1}}x_{i_1} - P_{\alpha g_{i_1}^{m_1 + 2}},\\
Q_{\beta g_{i_1}^{m_1}} &:= P_{\beta g_{i_1}^{m_1 + 1}} x_{i_1} - P_{\beta g_{i_1}^{m_1 + 2}},
\end{align*}
and the proof is similar to the subcase 2-2.
\end{itemize}

\noindent\fbox{case 3. $r=2$}\\
Let  $g = g_{i_1}^{m_1} g_{i_2}^{m_2}$ for $m_1, m_2 \in \Bbb{Z} \setminus \{0\}$.\\ 

\begin{itemize}
\item subcase 3-1. $m_1=1, m_2=1, 1\leq i_1 < i_2\leq 3$\\
In this case by the definition of $\bar I_G$ we have $P_{\alpha g_{i_1} g_{i_2}} - P_{\beta g_{i_1} g_{i_2}} \in\bar I_G$.\\

\item subcase 3-2. $m_1=1, m_2=1, 1\leq i_2 < i_1\leq 3$\\
Set
\begin{align*}
Q_{\alpha g_{i_1} g_{i_2}} &:= P_{\alpha} x_{i_1i_2} + x_{i_1} P_{\alpha g_{i_2}} + x_{i_2} P_{\alpha g_{i_1}} - P_{\alpha} x_{i_1} x_{i_2} - P_{\alpha g_{i_2} g_{i_1}},\\
Q_{\beta g_{i_1} g_{i_2}} &:= P_{\beta} x_{i_1i_2} + x_{i_1} P_{\beta g_{i_2}} + x_{i_2} P_{\beta g_{i_1}} - P_{\beta} x_{i_1} x_{i_2} - P_{\beta g_{i_2} g_{i_1}},
\end{align*}
so that we have $\Phi(Q_{\alpha g_{i_1} g_{i_2}}) = [\alpha g_{i_1} g_{i_2}]$ and $\Phi(Q_{\beta g_{i_1} g_{i_2}}) = [\beta g_{i_1} g_{i_2}]$, which follows from (4) in Lemma \ref{lem:skeineq}. Then we have
\begin{align*}
Q_{\alpha g_{i_1} g_{i_2}} - Q_{\beta g_{i_1} g_{i_2}} &= ( P_{\alpha} - P_{\beta} ) x_{i_1i_2} + ( P_{\alpha g_{i_2}} - P_{\beta g_{i_2}} ) x_{i_1}\\
&\quad + ( P_{\alpha g_{i_1}} - P_{\beta g_{i_1}}) x_{i_2} - ( P_{\alpha} - P_{\beta}) x_{i_1} x_{i_2}\\
&\quad - ( P_{\alpha g_{i_2} g_{i_1}} - P_{\beta g_{i_2} g_{i_1}}) \in\bar I_G.
\end{align*}

\item subcase 3-3. $m_1=1, m_2 >1$\\
We use induction on $m_2$. Set
\begin{align*}
Q_{\alpha g_{i_1}g_{i_2}^{m_2}} &:= P_{\alpha g_{i_1}g_{i_2}^{m_2 - 1}} x_{i_2} - P_{\alpha g_{i_1}g_{i_2}^{m_2 - 2}},\\
Q_{\beta g_{i_1}g_{i_2}^{m_2}} &:= P_{\beta g_{i_1}g_{i_2}^{m_2 - 1}} x_{i_2} - P_{\beta g_{i_1}g_{i_2}^{m_2 - 2}},
\end{align*}
so that we have $\Phi(Q_{\alpha g_{i_1}g_{i_2}^{m_2}}) = [\alpha g_{i_1}g_{i_2}^{m_2}]$ and $\Phi(Q_{\beta g_{i_1}g_{i_2}^{m_2}}) = [\beta g_{i_1}g_{i_2}^{m_2}]$, which follows from the defining relation $[gh]=[g]\otimes [h]  - [gh^{-1}]$ in the skein algebra. Then we have 
\begin{align*}
    Q_{\alpha g_{i_1}g_{i_2}^{m_2}} - Q_{\beta g_{i_1}g_{i_2}^{m_2}} = &( P_{\alpha g_{i_1}g_{i_2}^{m_2 - 1}} - P_{\beta g_{i_1}g_{i_2}^{m_2 - 1}} ) x_{i_2}\\
    &- ( P_{\alpha g_{i_1}g_{i_2}^{m_2 - 2}} - P_{\beta g_{i_1}g_{i_2}^{m_2 - 2}} ) \in\bar I_G,
\end{align*}
where the assertion for $m_2=2$ follows from the case 2 and subcases 3-1 and 3-2, that for $m_2=3$ follows from the subcases 3-1 and 3-2 and the case $m_2=2$, and that for $m_2\geq4$ follows from the induction on $m_2$. \\

\item subcase 3-4. $m_1=1, m_2 <0$\\
We use induction on $|m_2|$. Set 
\begin{align*}
Q_{\alpha g_{i_1}g_{i_2}^{m_2}} &:= P_{\alpha g_{i_1}g_{i_2}^{m_2 + 1}} x_{i_2} - P_{\alpha g_{i_1}g_{i_2}^{m_2 + 2}},\\
Q_{\beta g_{i_1}g_{i_2}^{m_2}} &:= P_{\beta g_{i_1}g_{i_2}^{m_2 + 1}} x_{i_2} - P_{\beta g_{i_1}g_{i_2}^{m_2 + 2}},
\end{align*}
and the proof is similar to the subcase 3-3.

\item subcase 3-5. $m_1>1, m_2\in\Bbb{Z}\setminus\{0\}$\\
We use induction on $m_1$. Set
\begin{align*}
Q_{\alpha g_{i_1}^{m_1}g_{i_2}^{m_2}} &:= P_{\alpha g_{i_1}^{m_1 - 1}g_{i_2}^{m_2}} x_{i_1} - P_{\alpha g_{i_1}^{m_1 - 2}g_{i_2}^{m_2}},\\
Q_{\beta g_{i_1}^{m_1}g_{i_2}^{m_2}} &:= P_{\beta g_{i_1}^{m_1 - 1}g_{i_2}^{m_2}} x_{i_1} - P_{\beta g_{i_1}^{m_1 - 2}g_{i_2}^{m_2}},
\end{align*} 
so that we have $\Phi(Q_{\alpha g_{i_1}^{m_1}g_{i_2}^{m_2}})=[\alpha g_{i_1}^{m_1}g_{i_2}^{m_2}]$ and $\Phi(Q_{\beta g_{i_1}^{m_1}g_{i_2}^{m_2}})=[\beta g_{i_1}^{m_1}g_{i_2}^{m_2}]$, which follows from (3) in Lemma \ref{lem:skeineq}. Then we have
\begin{align*}
    Q_{\alpha g_{i_1}^{m_1}g_{i_2}^{m_2}} - Q_{\beta g_{i_1}^{m_1}g_{i_2}^{m_2}} &= (P_{\alpha g_{i_1}^{m_1 - 1}g_{i_2}^{m_2}} - P_{\beta g_{i_1}^{m_1 - 1}g_{i_2}^{m_2}}) x_{i_1}\\
    &\quad - (P_{\alpha g_{i_1}^{m_1 - 2}g_{i_2}^{m_2}} - P_{\beta g_{i_1}^{m_1 - 2}g_{i_2}^{m_2}})\in\bar I_G,
\end{align*}
where the assertion for $m_1=2$ follows from case 2 and subcases 3-1 through 3-4, that for $m_1=3$ follows from subcases 3-1 through 3-4 and the case $m_1=2$, and that for $m_1\geq4$ follows from the induction on $m_1$.

\item subcase 3-6. $m_1<0, m_2\in\Bbb{Z}\setminus\{0\}$\\
We use induction on $|m_1|$. Set
\begin{align*}
Q_{\alpha g_{i_1}^{m_1}g_{i_2}^{m_2}} &:= P_{\alpha g_{i_1}^{m_1 + 1}g_{i_2}^{m_2}} x_{i_1} - P_{\alpha g_{i_1}^{m_1 + 2}g_{i_2}^{m_2}},\\
Q_{\beta g_{i_1}^{m_1}g_{i_2}^{m_2}} &:= P_{\beta g_{i_1}^{m_1 + 1}g_{i_2}^{m_2}} x_{i_1} - P_{\beta g_{i_1}^{m_1 + 2}g_{i_2}^{m_2}},
\end{align*} 
and the proof is similar to the subcase 3-5.

\end{itemize}

\noindent\fbox{case 4. $r=3, i_1\neq i_3$}\\
Let $g =g_{i_1}^{m_1}g_{i_2}^{m_2}g_{i_3}^{m_3}$ for $i_1\neq i_3$, $m_1, m_2, m_3 \in \Bbb{Z} \setminus \{0\}$. \\

\begin{itemize}
\item subcase 4-1. $m_1 = m_2 = m_3=1$
\\
In this case $(i_1, i_2, i_3)$ is a permutation of $(1,2,3)$.\\
	If $(i_1, i_2, i_3) = (1,2,3)$, then $P_{\alpha g_{1} g_{2} g_{3}} - P_{\beta g_{1} g_{2} g_{3}} \in\bar I_G$ by the definition of $\bar I_G$.\\

For $(i_1, i_2, i_3)=(1,3,2)$, we define 
\begin{align*}
Q_{\alpha g_1 g_3 g_2} &:= P_{\alpha g_1} x_{23} + x_3 P_{\alpha g_1 g_2} + x_2 P_{\alpha g_1 g_3} - P_{\alpha g_1} x_3 x_2 - P_{\alpha g_1 g_2 g_3},\\
Q_{\beta g_1 g_3 g_2} &:= P_{\beta g_1} x_{23} + x_3 P_{\beta g_1 g_2} + x_2 P_{\beta g_1 g_3} - P_{\beta g_1} x_3 x_2 - P_{\beta g_1 g_2 g_3}.
\end{align*}
Then we have $\Phi(Q_{\alpha g_1 g_3 g_2})=[\alpha g_1 g_3 g_2]$ and $\Phi(Q_{\beta g_1 g_3 g_2})=[\beta g_1 g_3 g_2]$ by Lemma \ref{lem:skeineq} (4). We have
\begin{align*}
Q_{\alpha g_1 g_3 g_2} - Q_{\beta g_1 g_3 g_2} &= ( P_{\alpha g_1} - P_{\beta g_1} ) x_{23} + x_3 (P_{\alpha g_1 g_2} - P_{\beta g_1 g_2} )\\
&\quad + x_2 ( P_{\alpha g_1 g_3} - P_{\beta g_1 g_3} )- ( P_{\alpha g_1} - P_{\beta g_1} ) x_3 x_2\\
&\quad - ( P_{\alpha g_1 g_2 g_3} - P_{\beta g_1 g_2 g_3}) \in\bar I_G,
\end{align*}
by the cases $1$-$3$ and  $(i_1, i_2, i_3)=(1,2,3)$.

We can prove the other $4$ cases similarly, using the following pairs of polynomials
$ Q_{\alpha g_{i_1} g_{1_2} g_{i_3}}$, $ Q_{\beta g_{i_1} g_{1_2} g_{i_3}}$.

For $(i_1,i_2,i_3)=(2,3,1)$, we define 
\begin{align*}
Q_{\alpha g_2 g_3 g_1} &:= P_{\alpha} x_{123} + x_{23} P_{\alpha g_1} + x_1 P_{\alpha g_2 g_3} - P_{\alpha} x_{23} x_1 - P_{\alpha g_1 g_2 g_3},\\
Q_{\beta g_2 g_3 g_1} &:= P_{\beta} x_{123} + x_{23} P_{\beta g_1} + x_1 P_{\beta g_2 g_3} - P_{\beta} x_{23} x_1 - P_{\beta g_1 g_2 g_3}.
\end{align*}

For $(i_1,i_2,i_3)=(2,1,3)$, we define
\begin{align*}
Q_{\alpha g_2 g_1 g_3} &:= P_{\alpha g_2} x_{13} + x_1 P_{\alpha g_2 g_3} + x_3 P_{\alpha g_2 g_1} - P_{\alpha g_2} x_1 x_3 - P_{\alpha g_2 g_3 g_1},\\
Q_{\beta g_2 g_1 g_3} &:= P_{\beta g_2} x_{13} + x_1 P_{\beta g_2 g_3} + x_3 P_{\beta g_2 g_1} - P_{\beta g_2} x_1 x_3 - P_{\beta g_2 g_3 g_1}.
\end{align*}

For  $(i_1,i_2,i_3)=(3,2,1)$,  we define 
\begin{align*}
Q_{\alpha g_3 g_2 g_1} &:= P_{\alpha} P_{g_3 g_2 g_1} + x_{23} P_{\alpha g_1} + x_1 P_{\alpha g_3 g_2} - P_{\alpha} x_{23} x_1 - P_{\alpha g_1 g_3 g_2},\\
Q_{\beta g_3 g_2 g_1} &:= P_{\beta} P_{g_3 g_2 g_1} + x_{23} P_{\beta g_1} + x_1 P_{\beta g_3 g_2} - P_{\beta} x_{23} x_1 - P_{\beta g_1 g_3 g_2}.
\end{align*}

For  $(i_1,i_2,i_3)=(3,1,2)$, we define 
\begin{align*}
Q_{\alpha g_3 g_1 g_2} &:= P_{\alpha g_3} x_{12} + x_1 P_{\alpha g_3 g_2} + x_2 P_{\alpha g_3 g_1} - P_{\alpha g_3} x_1 x_2 - P_{\alpha g_3 g_2 g_1},\\
Q_{\beta g_3 g_1 g_2} &:= P_{\beta g_3} x_{12} + x_1P_{\beta g_3 g_2} + x_2 P_{\beta g_3 g_1} - P_{\beta g_3} x_1 x_2 - P_{\beta g_3 g_2 g_1}.
\end{align*}

\item subcase 4-2. $m_1 = m_2 =1, m_3>1$\\
We use induction on $m_3$. Set
\begin{align*}
Q_{\alpha g_{i_1}g_{i_2}g_{i_3}^{m_3}} &:= P_{\alpha g_{i_1}g_{i_2}g_{i_3}^{m_3-1}}x_{i_3} - P_{\alpha g_{i_1}g_{i_2}g_{i_3}^{m_3-2}},\\
Q_{\beta g_{i_1}g_{i_2}g_{i_3}^{m_3}} &:= P_{\beta g_{i_1}g_{i_2}g_{i_3}^{m_3-1}}x_{i_3} - P_{\beta g_{i_1}g_{i_2}g_{i_3}^{m_3-2}},
\end{align*}
so that we have $\Phi(Q_{\alpha g_{i_1}g_{i_2}g_{i_3}^{m_3}})=[\alpha g_{i_1}g_{i_2}g_{i_3}^{m_3}]$ and $\Phi(Q_{\beta g_{i_1}g_{i_2}g_{i_3}^{m_3}})=[\beta g_{i_1}g_{i_2}g_{i_3}^{m_3}]$, which follows from the defining relation $gh=g\otimes h - gh^{-1}$ in the skein algebra. Then we have
\begin{align*}
Q_{\alpha g_{i_1}g_{i_2}g_{i_3}^{m_3}} - Q_{\beta g_{i_1}g_{i_2}g_{i_3}^{m_3}} &= (P_{\alpha g_{i_1}g_{i_2}g_{i_3}^{m_3-1}} - P_{\beta g_{i_1}g_{i_2}g_{i_3}^{m_3-1}})x_{i_3}\\
&\quad - (P_{\alpha g_{i_1}g_{i_2}g_{i_3}^{m_3-2}} - P_{\beta g_{i_1}g_{i_2}g_{i_3}^{m_3-2}})\in\bar I_G,
\end{align*}
where the assertion for $m_3=2$ follows from the case 3 and the subcase 4-1, that for $m_3=3$ follows from the subcase 4-1 and the case $m_3=2$, and that for $m_3\geq4$ follows from the induction on $m_3$.\\

\item subcase 4-3. $m_1 = m_2 =1, m_3<0$\\
We use induction on $|m_3|$. Set
\begin{align*}
Q_{\alpha g_{i_1}g_{i_2}g_{i_3}^{m_3}} &:= P_{\alpha g_{i_1}g_{i_2}g_{i_3}^{m_3+1}}x_{i_3} - P_{\alpha g_{i_1}g_{i_2}g_{i_3}^{m_3+2}},\\
Q_{\beta g_{i_1}g_{i_2}g_{i_3}^{m_3}} &:= P_{\beta g_{i_1}g_{i_2}g_{i_3}^{m_3+1}}x_{i_3} - P_{\beta g_{i_1}g_{i_2}g_{i_3}^{m_3+2}},
\end{align*}
and the proof is similar to the subcase 4-2.\\

\item subcase 4-4. $m_1 = 1, m_2>1, m_3\in\Bbb{Z}\setminus\{0\}$\\
We use induction on $m_2$. Set
\begin{align*}
Q_{\alpha g_{i_1}g_{i_2}^{m_2}g_{i_3}^{m_3}} &:= P_{\alpha g_{i_1}g_{i_2}^{m_2-1}g_{i_3}^{m_3}}x_{i_2} - P_{\alpha g_{i_1}g_{i_2}^{m_2-2}g_{i_3}^{m_3}},\\
Q_{\beta g_{i_1}g_{i_2}^{m_2}g_{i_3}^{m_3}} &:= P_{\beta g_{i_1}g_{i_2}^{m_2-1}g_{i_3}^{m_3}}x_{i_2} - P_{\beta g_{i_1}g_{i_2}^{m_2-2}g_{i_3}^{m_3}},
\end{align*}
so that we have $\Phi(Q_{\alpha g_{i_1}g_{i_2}^{m_2}g_{i_3}^{m_3}})=[\alpha g_{i_1}g_{i_2}^{m_2}g_{i_3}^{m_3}]$ and $\Phi(Q_{\beta g_{i_1}g_{i_2}^{m_2}g_{i_3}^{m_3}})=[\beta g_{i_1}g_{i_2}^{m_2}g_{i_3}^{m_3}]$, which follows from (3) in Lemma \ref{lem:skeineq}. Then we have
\begin{align*}
Q_{\alpha g_{i_1}g_{i_2}^{m_2}g_{i_3}^{m_3}} - Q_{\beta g_{i_1}g_{i_2}^{m_2}g_{i_3}^{m_3}} &= (P_{\alpha g_{i_1}g_{i_2}^{m_2-1}g_{i_3}^{m_3}} - P_{\beta g_{i_1}g_{i_2}^{m_2-1}g_{i_3}^{m_3}})x_{i_2}\\
&\quad - (P_{\alpha g_{i_1}g_{i_2}^{m_2-2}g_{i_3}^{m_3}} - P_{\beta g_{i_1}g_{i_2}^{m_2-2}g_{i_3}^{m_3}}) \in\bar I_G,
\end{align*}
where the assertion for $m_2=2$ follows from the case 3 and the subcases 4-1 through 4-3, that for $m_2=3$ follows from the subcases 4-1 through 4-3 and the case $m_2=2$, and that for $m_2\geq4$ follows from the induction on $m_2$.\\

\item subcase 4-5. $m_1 = 1, m_2<0, m_3\in\Bbb{Z}\setminus\{0\}$\\
We use induction on $|m_2|$. Set
\begin{align*}
Q_{\alpha g_{i_1}g_{i_2}^{m_2}g_{i_3}^{m_3}} &:= P_{\alpha g_{i_1}g_{i_2}^{m_2+1}g_{i_3}^{m_3}}x_{i_2} - P_{\alpha g_{i_1}g_{i_2}^{m_2+2}g_{i_3}^{m_3}},\\
Q_{\beta g_{i_1}g_{i_2}^{m_2}g_{i_3}^{m_3}} &:= P_{\beta g_{i_1}g_{i_2}^{m_2+1}g_{i_3}^{m_3}}x_{i_2} - P_{\beta g_{i_1}g_{i_2}^{m_2+2}g_{i_3}^{m_3}},
\end{align*}
and the proof is similar to the subcase 4-4.\\

\item subcase 4-6. $m_1>1, m_2\in\Bbb{Z}\setminus\{0\}, m_3\in\Bbb{Z}\setminus\{0\}$\\
We use induction on $m_1$. Set
\begin{align*}
Q_{\alpha g_{i_1}^{m_1}g_{i_2}^{m_2}g_{i_3}^{m_3}} &:= P_{\alpha g_{i_1}^{m_1-1}g_{i_2}^{m_2}g_{i_3}^{m_3}}x_{i_1} - P_{\alpha g_{i_1}^{m_1-2}g_{i_2}^{m_2}g_{i_3}^{m_3}},\\
Q_{\beta g_{i_1}^{m_1}g_{i_2}^{m_2}g_{i_3}^{m_3}} &:= P_{\beta g_{i_1}^{m_1-1}g_{i_2}^{m_2}g_{i_3}^{m_3}}x_{i_1} - P_{\beta g_{i_1}^{m_1-2}g_{i_2}^{m_2}g_{i_3}^{m_3}},
\end{align*}
so that we have $\Phi(Q_{\alpha g_{i_1}^{m_1}g_{i_2}^{m_2}g_{i_3}^{m_3}}) = [\alpha g_{i_1}^{m_1}g_{i_2}^{m_2}g_{i_3}^{m_3}]$ and $\Phi(Q_{\beta g_{i_1}^{m_1}g_{i_2}^{m_2}g_{i_3}^{m_3}}) = [\beta g_{i_1}^{m_1}g_{i_2}^{m_2}g_{i_3}^{m_3}]$, which follows from (3) in Lemma \ref{lem:skeineq}. Then we have
\begin{align*}
Q_{\alpha g_{i_1}^{m_1}g_{i_2}^{m_2}g_{i_3}^{m_3}} - Q_{\beta g_{i_1}^{m_1}g_{i_2}^{m_2}g_{i_3}^{m_3}} &= (P_{\alpha g_{i_1}^{m_1-1}g_{i_2}^{m_2}g_{i_3}^{m_3}} - P_{\beta g_{i_1}^{m_1-1}g_{i_2}^{m_2}g_{i_3}^{m_3}})x_{i_1}\\
&\quad - (P_{\alpha g_{i_1}^{m_1-2}g_{i_2}^{m_2}g_{i_3}^{m_3}} - P_{\beta g_{i_1}^{m_1-2}g_{i_2}^{m_2}g_{i_3}^{m_3}})\in\bar I_G,
\end{align*}
where the assertion for $m_1=2$ follows from the case 3 and subcases 4-1 through 4-5, that for $m_1=3$ follows from the subcases 4-1 through 4-5 and the case $m_1=2$, and that for $m_1\geq4$ follows from the induction on $m_1$.\\

\item subcase 4-7. $m_1<0, m_2\in\Bbb{Z}\setminus\{0\}, m_3\in\Bbb{Z}\setminus\{0\}$\\
We use induction on $|m_1|$. Set
\begin{align*}
Q_{\alpha g_{i_1}^{m_1}g_{i_2}^{m_2}g_{i_3}^{m_3}} &:= P_{\alpha g_{i_1}^{m_1+1}g_{i_2}^{m_2}g_{i_3}^{m_3}}x_{i_1} - P_{\alpha g_{i_1}^{m_1+2}g_{i_2}^{m_2}g_{i_3}^{m_3}},\\
Q_{\beta g_{i_1}^{m_1}g_{i_2}^{m_2}g_{i_3}^{m_3}} &:= P_{\beta g_{i_1}^{m_1+1}g_{i_2}^{m_2}g_{i_3}^{m_3}}x_{i_1} - P_{\beta g_{i_1}^{m_1+2}g_{i_2}^{m_2}g_{i_3}^{m_3}},
\end{align*}
and the proof is similar to the subcase 4-6.\\

\end{itemize}

\noindent\fbox{case 5. $r=3$ with $i_1= i_3$ or $r\geq4$}\\
Let  $g = g_{i_1}^{m_1}\cdots g_{i_r}^{m_r}$.
In this case there exist $i_{s_1} = i_{s_2}$ for some $1\leq s_1 < s_2 \leq r$. Let $X := g_{i_1}^{m_1} \ldots g_{i_{s_1}}^{m_{s_1}}, Y := g_{i_{s_1+1}}^{m_{s_1+1}} \ldots g_{i_{s_2}}^{m_{s_2}}, Z:=g_{i_{s_2+1}}^{m_{s_2+1}} \ldots g_{i_{r}}^{m_{r}}$ (i.e., $g = XYZ$). We assume that $Z=e$ when $s_2=r$.
Note that both $XZ$ and $XY^{-1}Z$ have the lengths  less than $XYZ$. 
For example for $r=3$,  we have $X=g_{i_1}^{m_1}$, $Y=g_{i_2}^{m_2}g_{i_1}^{m_3}$, $Z=e$, and we have  $XZ=X=g_{i_1}^{m_1}$ and $XY^{-1}Z=g_{i_1}^{m_1-m_3}g_{i_2}^{-m_2}.$ 
Set
\begin{align*}
Q_{\alpha XYZ} &:= P_{\alpha XZ} P_{Y} - P_{\alpha XY^{-1}Z},\\
Q_{\beta XYZ} &:= P_{\beta XZ} P_{Y} - P_{\beta XY^{-1}Z},
\end{align*}
so that  $\Phi(Q_{\alpha XYZ})=[\alpha XYZ]$ and $\Phi(Q_{\beta XYZ})=[\beta XYZ]$, which follows from Lemma \ref{lem:skeineq} (3). Then we have
\begin{align*}
Q_{\alpha XYZ} - Q_{\beta XYZ} = (P_{\alpha XZ} - P_{\beta XZ}) P_{Y} - (P_{\alpha XY^{-1}Z} - P_{\beta XY^{-1}Z}) \in\bar I_G,
\end{align*}
by the induction on $r$.  
\end{proof}

\section{Skein algebra of Borromean rings complement in $S^3$}\label{SS}

In this section we consider the skein algebra of the fundamental group $\pi_1M_B$ of the Borromean rings complement $M_B$. We give explicit generators of $I_{\pi_1M_B}$.

\subsection{Fundamental group of  Borromean rings complement}\label{SS_1}
Let $M_B$ be the Borromean rings complement in $S^3$. We will use the following presentation of $\pi_1M_B$.

\begin{lem}\label{pi1ofBorromean}
We have
\begin{equation*}
\pi_1 M_B = \langle\  g_1, g_2, g_3\ |\   \alpha = \beta, \ \gamma = \delta\ \rangle,
\end{equation*}
where $\alpha = g_3g_2^{-1}g_1g_2g_1^{-1}, \beta = g_2^{-1}g_1g_2 g_1^{-1}g_3, \gamma=g_2g_1^{-1}g_3g_1g_3^{-1}, \delta=g_1^{-1}g_3g_1g_3^{-1}g_2$.\\
\end{lem}
\begin{figure}
\centering
\begin{picture}(300,190)
\put(30,20){\includegraphics[width=5.8cm,clip]{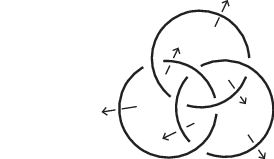}}
\put(113,80){$1$}
\put(145,78){$2$}
\put(183,77){$3$}
\put(127,47){$4$}
\put(163,43){$5$}
\put(147,13){$6$}
\put(169,110){$g_1$}
\put(88,40){$g_2$}
\put(180,15){$g_3$}
\put(178,49){$g_4$}
\put(130,90){$g_5$}
\put(117,35){$g_6$}
\end{picture}
\caption{Wirtinger presentation of the Borromen rings}\label{fig:borromean}
\end{figure}
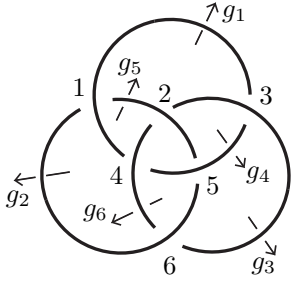 
\begin{proof}
Let $g_1, g_2, g_3, g_4, g_5, g_6$ be the meridians of the link as shown in Figure \ref{fig:borromean}.
By the Wirtinger presentation of the Borromean rings, we have
\begin{equation*}
\pi_1M_B = \langle g_1, g_2, g_3, g_4, g_5, g_6\ |\ \mathbf{r}_1, \mathbf{r}_2,\mathbf{r}_3, \mathbf{r}_4, \mathbf{r}_5, \mathbf{r}_6\rangle
\end{equation*}
where $\mathbf{r}_i$, $i=1,\ldots,6$, is the relation obtained from the $i$-th crossing as below;
\begin{align*}
\mathbf{r}_1 : g_2g_1^{-1}g_5^{-1}g_1 = e,\\
\mathbf{r}_2 : g_6g_5g_3^{-1}g_5^{-1} = e, \\
\mathbf{r}_3 : g_1g_3^{-1}g_4^{-1}g_3 = e, \\
\mathbf{r}_4 : g_4g_6g_1^{-1}g_6^{-1} = e, \\
\mathbf{r}_5 : g_5g_4g_2^{-1}g_4^{-1} = e, \\
\mathbf{r}_6 : g_3g_2^{-1}g_6^{-1}g_2 = e.  
\end{align*}

By the relations $\mathbf{r}_1, \mathbf{r}_3, \mathbf{r}_6$, we have
\begin{align*}
g_5 = g_1g_2g_1^{-1},\quad
g_4 = g_3g_1g_3^{-1},\quad
g_6 = g_2g_3g_2^{-1},
\end{align*}
respectively. Thus $\pi_1M_B$ is presented by generators $g_1, g_2, g_3$ and relations 

\begin{align*}
\mathbf{r'}_2& : (g_2g_3g_2^{-1})(g_1g_2g_1^{-1})g_3^{-1}(g_1g_2^{-1}g_1^{-1}) = e,\\
\mathbf{r'}_4& : (g_3g_1g_3^{-1})(g_2g_3g_2^{-1})g_1^{-1}(g_2g_3^{-1}g_2^{-1}) = e,\\
\mathbf{r'}_5& : (g_1g_2g_1^{-1})(g_3g_1g_3^{-1})g_2^{-1}(g_3g_1^{-1}g_3^{-1}) = e.
\end{align*}

In fact the relation $\mathbf{r'}_4$ is derived from $\mathbf{r'}_2, \mathbf{r'}_5$ as follows. By  $\mathbf{r'}_2$ and $\mathbf{r'}_5$ we have 
\begin{align*}
 (g_2g_3g_2^{-1})(g_1g_2g_1^{-1})g_3^{-1}&=(g_1g_2g_1^{-1})\\
 &= (g_3g_1g_3^{-1})g_2(g_3g_1^{-1}g_3^{-1}),
\end{align*}
thus we have
$g_3g_1g_3^{-1}g_2g_3 = g_2g_3g_2^{-1}g_1g_2$,
which is equivalent to $\mathbf{r'}_4$.
Note that $\mathbf{r'}_2$ is equivalent to $\alpha = \beta$, and $\mathbf{r'}_5 $ is equivalent to $\gamma = \delta$.

Thus we have the assertion.

\end{proof}

\subsection{Proof of Theorem \ref{Theorem:3}}\label{SS_2}
In this subsection we prove Theorem \ref{Theorem:3}.
Recall the polynomials $Q_{\alpha g, \beta g}, Q_{\gamma g, \delta g}$ in  Theorem  \ref{Theorem:3}. 
We define $Q_{\alpha, \beta}, Q_{\gamma, \delta}, Q_{\gamma g_2, \delta g_2}, Q_{\alpha g_3, \beta g_3}$, which are not on the list in Theorem \ref{Theorem:3}, as  copies of the zero polynomial, and consider the ideal
$$\hat I_{\pi_1M_B} := \langle K, Q_{\alpha g, \beta g}, Q_{\gamma g, \delta g}\ |\ g=g_1^{i_1}g_2^{i_2}g_3^{i_3}, 0\leq i_1,i_2,i_3\leq1\rangle \subset \Bbb{C}[x_1,\ldots,x_{123}].$$
Recall that  we have
$$I_{\pi_1M_B} = \bar I_{\pi_1M_B} = \langle K, P_{\alpha g} - P_{\beta g}, P_{\gamma g} - P_{\delta g}\ |\ g=g_1^{i_1}g_2^{i_2}g_3^{i_3}, 0\leq i_1,i_2,i_3\leq1\rangle$$
by Theorem \ref{Theorem:2}, and thus it suffices to show that $\bar I_{\pi_1M_B} = \hat I_{\pi_1M_B}$. 

We reduce Theorem \ref{Theorem:3} to the following lemma.
\begin{lem}\label{dev}
$\Phi(Q_{\alpha g, \beta g}) = [\alpha g] - [\beta g]$ and $\Phi(Q_{\gamma g, \delta g}) = [\gamma g] - [\delta g]$ for each $g\in\{g_1^{i_1}g_2^{i_2}g_3^{i_3}\ |\ 0\leq i_1,i_2,i_3\leq1\}$. 
\end{lem}

\begin{proof}[Proof of Theorem \ref{Theorem:3}  assuming Lemma \ref{dev}]
By Lemma \ref{dev} we have $\bar I_{\pi_1M_B} = \hat I_{\pi_1M_B}$ because
\begin{align*}
(P_{\alpha g} - P_{\beta g}) - Q_{\alpha g, \beta g} &\in\ker\Phi=\langle K \rangle \subset \hat I_{\pi_1M_B}\cap\bar I_{\pi_1M_B},\\
(P_{\gamma g} - P_{\delta g}) - Q_{\gamma g, \delta g} &\in\ker\Phi=\langle K \rangle \subset \hat I_{\pi_1M_B}\cap\bar I_{\pi_1M_B}.
\end{align*}
This completes the proof.
\end{proof}

In what follows, we prove Lemma \ref{dev}. We will use the following formulae.

\begin{lem}\label{cyc}
We have
\begin{align*}
\tag{a}[g_i^{2}]&=[g_i]^2-2, \quad \text{for } i=1,2,3,\\
\tag{b}[g_ig_j^{-1}]&=[g_i]\otimes[g_j]-[g_ig_j],\quad \text{for } i,j=1,2,3, \ i\not =j,\\
\tag{c1}[g_1^{-1}g_3g_2] 
&= - [g_2]\otimes[g_1g_3] - [g_3]\otimes[g_1g_2] + [g_1]\otimes[g_2]\otimes[g_3] + [g_1g_2g_3],\\
\tag{c2}[g_2^{-1}g_1g_3] 
&= -[g_3]\otimes[g_1g_2] - [g_1]\otimes[g_2g_3] + [g_1]\otimes[g_2]\otimes[g_3] + [g_1g_2g_3],\\
\tag{c3}[g_3^{-1}g_2g_1] 
&= -[g_1]\otimes[g_2g_3] - [g_2]\otimes[g_1g_3] + [g_1]\otimes[g_2]\otimes[g_3] + [g_1g_2g_3].
\end{align*}
\end{lem}

\begin{proof}
(a1) and (a2) follow from the defining relations $[gh]=[g]\otimes [h]  - [gh^{-1}]$ $(=[g]\otimes [h]  - [g^{-1}h])$ and $[e]=2$ of the skein algebra. 
For (b1), using Lemma \ref{lem:skeineq} (4), we have
\begin{align*}
[g_1^{-1}g_3g_2] 
&= [g_1^{-1}]\otimes[g_3g_2] - [g_1g_3g_2]\\
&= [g_1]\otimes[g_3g_2] - ([g_1]\otimes[g_3g_2] + [g_3]\otimes[g_1g_2] + [g_2]\otimes[g_1g_3]\\
&\quad - [g_1]\otimes[g_3]\otimes[g_2] - [g_1g_2g_3])\\
&= - [g_2]\otimes[g_1g_3] - [g_3]\otimes[g_1g_2] + [g_1]\otimes[g_2]\otimes[g_3] + [g_1g_2g_3].
\end{align*}
We can prove (b2) and (b3) similarly.
\end{proof}

\begin{proof}[Proof of Lemma \ref{dev}]
We use the defining relation $[e]=2$ and  $[gh]=[g]\otimes [h]  - [gh^{-1}]$ $(=[g]\otimes [h]  - [g^{-1}h])$ of the skein algebra and Lemma \ref{lem:skeineq} freely.
\\

\noindent\fbox{$Q_{\alpha,\beta}$ and $Q_{\gamma, \delta}$}\\

We have $\Phi(Q_{\alpha,\beta}) = 0 = [\alpha] - [\beta]$ and $\Phi(Q_{\gamma, \delta}) = 0 = [\gamma] - [\delta]$.\\

\noindent\fbox{$Q_{\alpha g_1,\beta g_1}$ and $Q_{\gamma g_1, \delta g_1}$}\\

We have
\begin{align*}
[\alpha g_1] - [\beta g_1] &= [g_3{g_2}^{-1} g_1g_2g_1^{-1}g_1] - [{g_2}^{-1}g_1g_2{g_1}^{-1}g_3g_1]\\&= [g_3\cdot{g_2}^{-1}\cdot (g_1g_2)] - [({g_2}^{-1}g_1)\cdot(g_2{g_1}^{-1}g_3g_1)]\\
&= \Big([g_3]\otimes[{g_2}^{-1}g_1g_2] + [{g_2}^{-1}]\otimes[g_3g_1g_2] + [g_1g_2]\otimes[g_3\cdot {g_2}^{-1}]\\
&\quad - [g_3]\otimes[g_2^{-1}]\otimes[g_1g_2] - [g_3\cdot (g_1g_2)\cdot{g_2}^{-1}] \Big)\\
&\quad - \Big([{g_2}^{-1}\cdot g_1]\otimes[g_2){g_1}^{-1}(g_3g_1] - [({g_1}^{-1}{g_2})\cdot (g_2{g_1}^{-1}g_3g_1)]\Big)\\
&= \Big([g_3]\otimes[g_1] + [{g_2}]\otimes[g_1g_2g_3] + [g_1g_2]\otimes\left([g_3]\otimes[{g_2}] - [g_3g_2]\right)\\
&\quad - [g_3]\otimes[g_2]\otimes[g_1g_2] - [g_3g_1]\Big)\\
&\quad - \Big(\left([g_2]\otimes[g_1] -[g_2g_1]\right)\otimes \left([g_1^{-1}]\otimes[g_3g_1g_2] - [g_1g_3g_1g_2]\right) - [g_2^2g_1^{-1}g_3]\Big)\\
&= [g_3]\otimes[g_1] + [{g_2}]\otimes[g_1g_2g_3] - [g_1g_2]\otimes[g_2g_3] - [g_1g_3]\\
&\quad - \left([g_1]\otimes[g_2] -[g_1g_2]\right)\otimes \left([g_1]\otimes[g_1g_2g_3] - [g_1g_3g_1g_2]\right) + [g_2^2g_1^{-1}g_3].
\end{align*}
Here we have
\begin{align*}
[(g_1g_3) \cdot (g_1g_2)]
&= [g_1g_3]\otimes[g_1g_2] - [{g_3}^{-1}{g_1}^{-1}g_1g_2] \\
&= [g_1g_3]\otimes[g_1g_2] - [g_3^{-1} g_2],\\
[g_2^2{g_1}^{-1}g_3] &= [g_2\cdot (g_2{g_1}^{-1}g_3)]
\\
&=[g_2]\otimes[g_2g_1^{-1}g_3] - [g_1^{-1}g_3].\\
\end{align*}
Using the above identities and Lemma \ref{cyc} (b), (c1), we have
\begin{align*}
[\alpha g_1] - [\beta g_1] &= -2[g_1g_2]\otimes[g_2g_3] + 2[g_2]\otimes[g_1g_2g_3] - [g_1]^{\otimes2}\otimes[g_2]\otimes[g_1g_2g_3]\\
&\quad + [g_1]\otimes[g_2]\otimes[g_1g_2]\otimes[g_1g_3] + [g_1]\otimes[g_2]\otimes[g_2g_3] + [g_1]\otimes[g_1g_2]\otimes[g_1g_2g_3]\\
&\quad - [g_1g_2]^{\otimes2}\otimes[g_1g_3] - [g_2]^{\otimes2}\otimes[g_1g_3]
\\
&=\Phi(Q_{\alpha g_1, \beta g_1}).
\end{align*}

On the other hand, we have
\begin{align*}
[\gamma g_1] - [\delta g_1] &= [(g_2g_1^{-1}g_3g_1)\cdot (g_3^{-1}g_1)] - [g_1^{-1}g_3g_1g_3^{-1}g_2g_1]\\
&= \left([g_2g_1^{-1}g_3g_1]\otimes[g_3^{-1}g_1] - [g_2g_1^{-1}g_3^2]\right) - [g_3g_1g_3^{-1}g_2]\\
&= \left([g_2g_1^{-1}g_3g_1]\otimes[g_3^{-1}g_1] - [g_3]\otimes[g_2g_1^{-1}g_3] + [g_2g_1^{-1}]\right) - [g_3g_1g_3^{-1}g_2].
\end{align*}
Here we have
\begin{align*}
[g_2\cdot g_1^{-1}\cdot(g_3g_1)] &= [g_2]\otimes[g_3] + [g_1^{-1}]\otimes[g_2g_3g_1] + [g_3g_1]\otimes[g_2g_1^{-1}] - [g_2]\otimes[g_1^{-1}]\otimes[g_3g_1] - [g_2g_3],\\
[(g_3g_1) \cdot g_3^{-1}\cdot g_2] &= [g_3g_1]\otimes[g_3^{-1}g_2] + [g_3^{-1}]\otimes[g_3g_1g_2] + [g_2]\otimes[g_1] - [g_3g_1]\otimes[g_3^{-1}]\otimes[g_2] - [g_1g_2].
\end{align*}
Using the above identities  and Lemma \ref{cyc} (b), (c1), by a straight calculation  we have
\begin{align*}
[\gamma g_1] - [\delta g_1] &= 2[g_1g_3]\otimes[g_2g_3]- 2[g_3]\otimes[g_1g_2g_3] + [g_1]^{\otimes2}\otimes[g_3]\otimes[g_1g_2g_3]
\\ 
&\quad- [g_1]\otimes[g_3]\otimes[g_1g_2]\otimes[g_1g_3]
-[g_1]\otimes[g_1g_3]\otimes[g_1g_2g_3] 
- [g_1]\otimes[g_3]\otimes[g_2g_3]\\
&\quad + [g_1g_2]\otimes[g_1g_3]^{\otimes2} + 
 [g_1g_2]\otimes[g_3]^{\otimes2}
\\
&=\Phi(Q_{\gamma g_1,\delta g_1}).
\end{align*}

\noindent\fbox{$Q_{\alpha g_2, \beta g_2}$ and $Q_{\gamma g_2, \delta g_2}$}\\

We have
\begin{align*}
[\alpha g_2] - [\beta g_2] &= [(g_3g_2^{-1}g_1g_2)\cdot (g_1^{-1}g_2)] - [g_2^{-1}g_1g_2g_1^{-1}g_3g_2]\\
&= \left([g_3g_2^{-1}g_1g_2]\otimes[g_1^{-1}g_2] - [g_3g_2^{-1}g_1^2]\right) - [g_1g_2g_1^{-1}g_3].
\end{align*}
Here we have
\begin{align*}
[g_3\cdot g_2^{-1}\cdot (g_1g_2)]&= [g_3]\otimes[g_1]+[g_2^{-1}]\otimes[g_3g_1g_2] + [g_1g_2]\otimes[g_3g_2^{-1}] - [g_3]\otimes[g_2^{-1}]\otimes[g_1g_2] - [g_3g_1],\\
[g_3g_2^{-1}g_1^2] 
&= [(g_3g_2^{-1}g_1)\cdot g_1]\\ &=[g_3g_2^{-1}g_1]\otimes[g_1] - [g_3g_2^{-1}]\\
[(g_1g_2)\cdot g_1^{-1}\cdot g_3] &= [g_1g_2]\otimes[g_1^{-1}g_3] + [g_1^{-1}]\otimes[g_1g_2g_3] + [g_3]\otimes[g_2] - [g_1g_2]\otimes[g_1^{-1}]\otimes[g_3] - [g_2g_3].
\end{align*}
Using the above identities and Lemma \ref{cyc} (b), (c2), by a straight calculation we have
\begin{align*}
[\alpha g_2] - [\beta g_2] &= - [g_1]\otimes[g_2]\otimes[g_1g_2]\otimes[g_2g_3] + [g_1g_2]^{\otimes2}\otimes[g_2g_3] + [g_1]^{\otimes2}\otimes[g_2g_3]\\
&\quad + [g_1]\otimes[g_2]^{\otimes2}\otimes[g_1g_2g_3] - [g_2]\otimes[g_1g_2]\otimes[g_1g_2g_3] -2[g_1]\otimes[g_1g_2g_3]\\
&\quad - [g_1]\otimes[g_2]\otimes[g_1g_3] + 2[g_1g_2]\otimes[g_1g_3]
\\
&=\Phi(Q_{\alpha g_2, \beta g_2}).
\end{align*}

On the other hand, we have
$$[\gamma g_2] - [\delta g_2] = [g_2g_1^{-1}g_3g_1g_3^{-1}g_2] - [g_1^{-1}g_3g_1g_3^{-1}g_2g_2]= 0=\Phi(Q_{\gamma g_2, \delta g_2}).$$

\noindent\fbox{$Q_{\alpha g_3, \beta g_3}$ and $Q_{\gamma g_3, \delta g_3}$}\\

We have
$$[\alpha g_3] - [\beta g_3] =  [g_3g_2^{-1}g_1g_2g_1^{-1}g_3] - [g_2^{-1}g_1g_2g_1^{-1}g_3g_3] = 0=\Phi(Q_{\alpha g_3, \beta g_3}).$$

On the other hand, we have
\begin{align*}
[\gamma g_3] - [\delta g_3] &= [g_2g_1^{-1}g_3g_1g_3^{-1}g_3] - [(g_1^{-1}g_3)\cdot (g_1g_3^{-1}g_2g_3)]\\
&= [g_2g_1^{-1}g_3g_1] - \left( [g_1^{-1}g_3]\otimes[g_1g_3^{-1}g_2g_3] - [g_1\cdot (g_1g_3^{-1}g_2)]\right)\\
&= [g_2g_1^{-1}g_3g_1] - \left([g_1^{-1}g_3]\otimes[g_1g_3^{-1}g_2g_3] - [g_1]\otimes[g_1g_3^{-1}g_2] + [g_3^{-1}g_2]\right).
\end{align*}
Here we have
\begin{align*}
[g_2\cdot g_1^{-1}\cdot (g_3g_1)] &= [g_2]\otimes[g_3] + [g_1^{-1}]\otimes[g_2g_3g_1] + [g_3g_1]\otimes[g_2g_1^{-1}] - [g_2]\otimes[g_1^{-1}]\otimes[g_3g_1] - [g_2g_3],\\
[g_1g_3^{-1}g_2g_3] &= [(g_3g_1)\cdot g_3^{-1}\cdot g_2]\\
&= [g_3g_1]\otimes[g_3^{-1}g_2] + [g_3^{-1}]\otimes[g_3g_1g_2] + [g_2]\otimes[g_1] - [g_3g_1]\otimes[g_3^{-1}]\otimes[g_2] - [g_1g_2].
\end{align*} 
Using the above identities and Lemma \ref{cyc} (b), (c3), by a straight calculation we have
\begin{align*}
[\gamma g_3] - [\delta g_3] &= [g_1]\otimes[g_3]\otimes[g_1g_3]\otimes[g_2g_3] - [g_1g_3]^{\otimes2}\otimes[g_2g_3] - [g_1]^{\otimes2}\otimes[g_2g_3]\\
&\quad -  [g_1]\otimes[g_3]^{\otimes2}\otimes[g_1g_2g_3] + [g_3]\otimes[g_1g_3]\otimes[g_1g_2g_3]\\
&\quad + 2[g_1]\otimes[g_1g_2g_3] + [g_1]\otimes[g_3]\otimes[g_1g_2] - 2[g_1g_2]\otimes[g_1g_3]
\\
&=\Phi(Q_{\gamma g_3, \delta g_3}).
\end{align*}

\noindent\fbox{$Q_{\alpha g_1g_2, \beta g_1g_2}$ and $Q_{\gamma g_1g_2, \delta g_1g_2}$}\\

We have
\begin{align*}
[\alpha g_1g_2] - [\beta g_1g_2] &= [g_3g_2^{-1}g_1g_2g_1^{-1}g_1g_2] - [g_2^{-1}g_1g_2g_1^{-1}g_3g_1g_2]\\
&= [(g_3g_2^{-1}g_1g_2)\cdot g_2] - [(g_1g_2g_1^{-1}g_3)\cdot g_1]\\
&= \left([g_3g_2^{-1}g_1g_2]\otimes[g_2] - [g_3g_2^{-1}g_1]\right) - \left([g_1g_2g_1^{-1}g_3]\otimes[g_1] - [g_2g_1^{-1}g_3]\right).
\end{align*}
Here we have
\begin{align*}
[g_3\cdot g_2^{-1}\cdot (g_1g_2)] &= [g_3]\otimes[g_1] + [g_2^{-1}]\otimes[g_3g_1g_2] + [g_1g_2]\otimes[g_3g_2^{-1}] - [g_3]\otimes[g_2^{-1}]\otimes[g_1g_2] - [g_3g_1],\\
[(g_1g_2)\cdot g_1^{-1}\cdot g_3] &= [g_1g_2]\otimes[g_1^{-1}g_3] + [g_1^{-1}]\otimes[g_1g_2g_3] + [g_3]\otimes[g_2] - [g_1g_2]\otimes[g_1^{-1}]\otimes[g_3] -[g_2g_3].
\end{align*}
Using the above identities and Lemma \ref{cyc} (b), (c1) and (c2), by a straight calculation  we have
\begin{align*}
[\alpha g_1g_2] - [\beta g_1g_2] &= -[g_1]^{\otimes2}\otimes[g_1g_2g_3] + [g_2]^{\otimes2}\otimes[g_1g_2g_3] + [g_1]\otimes[g_1g_2]\otimes[g_1g_3]\\
&\quad - [g_2]\otimes[g_1g_2]\otimes[g_2g_3] - 2[g_2]\otimes[g_1g_3] + 2[g_1]\otimes[g_2g_3]
\\
&=\Phi(Q_{\alpha g_1g_2, \beta g_1g_2}). 
\end{align*}

On the other hand we have
\begin{align*}
[\gamma g_1g_2] -  [\delta g_1g_2] &= [(g_2g_1^{-1}g_3g_1g_3^{-1}g_1) \cdot g_2] - [
(g_1^{-1}g_3g_1g_3^{-1}g_2) \cdot (g_1g_2)]\\
&= \left([(g_2g_1^{-1}g_3g_1)\cdot (g_3^{-1}g_1)]\otimes[g_2] - [g_1^{-1}g_3g_1g_3^{-1}g_1]\right)\\
&\quad - \left([g_1^{-1}g_3g_1g_3^{-1}g_2]\otimes[g_1g_2] - [g_1^{-1}g_3g_1g_3^{-1}g_1^{-1}]\right)\\
&= \Big(\left([g_2g_1^{-1}g_3g_1]\otimes[g_3^{-1}g_1] - [g_2g_1^{-1}g_3^2]\right)\otimes[g_2] - [g_1]\Big)\\
&\quad -\left( [g_1^{-1}g_3g_1g_3^{-1}g_2]\otimes[g_1g_2] - [g_1^{-1}g_3g_1g_3^{-1}g_1^{-1}]\right).
\end{align*}
Here we have
\begin{align*}
[g_2g_1^{-1}g_3g_1] &= [(g_1g_2)\cdot g_1^{-1}\cdot g_3]\\
&= [g_1g_2]\otimes[g_1^{-1}g_3] + [g_1^{-1}]\otimes[g_1g_2g_3] + [g_3]\otimes[g_2]\\
&\quad - [g_1g_2]\otimes[g_1^{-1}]\otimes[g_3] -[g_2g_3],\\
[g_2g_1^{-1}g_3^2] &= [(g_2g_1^{-1}g_3)\cdot g_3] =[g_2g_1^{-1}g_3]\otimes[g_3] - [g_2g_1^{-1}],\\
[(g_1^{-1}g_3)\cdot g_1\cdot(g_3^{-1}g_2)] &= [g_1^{-1}g_3]\otimes[g_1g_3^{-1}g_2] + [g_1]\otimes[g_1^{-1}g_2] + [g_3^{-1}g_2]\otimes[g_3]\\
&\quad - [g_1^{-1}g_3]\otimes[g_1]\otimes[g_3^{-1}g_2] - [g_2],\\
[(g_1^{-1}g_3)\cdot g_1\cdot(g_3^{-1}g_1^{-1})] &= [g_1^{-1}g_3]\otimes[g_1g_3^{-1}g_1^{-1}] + [g_1]\otimes[g_1^{-2}] + [g_3^{-1}g_1^{-1}]\otimes[g_3]\\
&\quad - [g_1^{-1}g_3]\otimes[g_1]\otimes[g_3^{-1}g_1^{-1}] -  [g_1^{-1}],
\end{align*}
Using the above identities and Lemma \ref{cyc} (a), (b), (c1), (c3), by a straight calculation  we have
\begin{align*}
[\gamma g_1g_2] - [\delta g_1g_2] &= [g_1]^{\otimes3} + [g_1]\otimes[g_3]^{\otimes2} + [g_1]\otimes[g_1g_3]^{\otimes2} - [g_1]^{\otimes2}\otimes[g_3]\otimes[g_1g_3] -4 [g_1]\\
&\quad + [g_1]^{\otimes2}\otimes[g_2]\otimes[g_3]\otimes[g_1g_2g_3] - [g_1]\otimes[g_2]\otimes[g_1g_3]\otimes[g_1g_2g_3]\\
&\quad - [g_1]\otimes[g_2]\otimes[g_3]\otimes[g_2g_3] - [g_1]\otimes[g_3]\otimes[g_1g_2]\otimes[g_1g_2g_3]\\
&\quad +  [g_1g_2]\otimes[g_1g_3]\otimes[g_1g_2g_3] - [g_1]^{\otimes2}\otimes[g_2]\otimes[g_1g_2]+ [g_1]\otimes[g_1g_2]^{\otimes2}\\
&\quad + [g_3]\otimes[g_1g_2]\otimes[g_2g_3] - [g_2]\otimes[g_3]\otimes[g_1g_2g_3] + [g_2]\otimes[g_1g_3]\otimes[g_2g_3]\\
&\quad + [g_1]\otimes[g_2]^{\otimes2}
\\
&=\Phi(Q_{\gamma g_1g_2,  \delta g_1g_2}).
\end{align*}

\noindent\fbox{$Q_{\alpha g_1g_3,\beta g_1g_3}$ and $Q_{\gamma g_1g_3,\delta g_1g_3}$}\\

We have
\begin{align*}
[\alpha g_1g_3] - [\beta g_1g_3] &= [g_3\cdot(g_2^{-1}g_1g_2g_1^{-1}g_1g_3)] - [(g_2^{-1}g_1g_2g_1^{-1}g_3)\cdot(g_1g_3)]\\
&= \left([g_3]\otimes[g_2^{-1}g_1g_2g_3] - [g_1]\right)- \left([g_2^{-1}g_1g_2 g_1^{-1}g_3]\otimes[g_1g_3] - [g_2^{-1}g_1g_2 g_1^{-1}g_1^{-1}]\right).
\end{align*}
Here we have
\begin{align*}
[g_2^{-1}\cdot g_1\cdot(g_2g_3)] &= [g_2^{-1}]\otimes[g_1g_2g_3] + [g_1]\otimes[g_3] + [g_2g_3]\otimes[g_2^{-1}g_1]\\
&\quad - [g_2^{-1}]\otimes[g_1]\otimes[g_2g_3] - [g_1g_3],\\
[(g_2^{-1}g_1)\cdot g_2\cdot(g_1^{-1}g_3)] &= [g_2^{-1}g_1]\otimes[g_2g_1^{-1}g_3] + [g_2]\otimes[g_2^{-1}g_3] + [g_1^{-1}g_3]\otimes[g_1]\\
&\quad - [g_2^{-1}g_1]\otimes[g_2]\otimes[g_1^{-1}g_3] - [g_3],\\
[(g_2^{-1}g_1g_2 g_1^{-1})\cdot g_1^{-1}] &= [g_2^{-1}\cdot(g_1g_2)\cdot g_1^{-1}]\otimes[g_1^{-1}] - [g_2^{-1}g_1g_2]\\
&= \left([g_2^{-1}]\otimes[g_2] + [g_1g_2]\otimes[g_2^{-1}g_1^{-1}] + [g_1^{-1}]\otimes [g_1] - [g_2^{-1}]\otimes[g_1g_2]\otimes[g_1^{-1}] - [e]\right)\\
&\quad\otimes[g_1^{-1}] - [g_1]\\
&= \left([g_2]^{\otimes2} + [g_1g_2]^{\otimes2} + [g_1]^{\otimes2} - [g_2]\otimes[g_1g_2]\otimes[g_1] - 2\right)\otimes[g_1] - [g_1].
\end{align*}
Using the above identities and Lemma \ref{cyc} (b), (c1), by a straight calculation we have
\begin{align*}
[\alpha g_1g_3] - [\beta g_1g_3] &= -4[g_1]+[g_1]^{\otimes3} + [g_1]\otimes[g_1g_2]^{\otimes2} + [g_1g_2]\otimes[g_1g_3]\otimes[g_1g_2g_3]\\
&\quad + [g_1]\otimes[g_1g_3]^{\otimes2} - [g_1]^{\otimes2}\otimes[g_2]\otimes[g_1g_2]\\
&\quad - [g_1]\otimes[g_2]\otimes[g_1g_3]\otimes[g_1g_2g_3] + [g_1]\otimes[g_2]^{\otimes2}\\
&\quad + [g_2]\otimes[g_1g_3]\otimes[g_2g_3] - [g_1]^{\otimes2}\otimes[g_3]\otimes[g_1g_3]\\
&\quad - [g_3]\otimes[g_1g_3]\otimes[g_1g_2]^{\otimes2} + [g_2]\otimes[g_3]\otimes[g_1g_2g_3]\\
&\quad + [g_1]\otimes[g_2]\otimes[g_3]\otimes[g_1g_2]\otimes[g_1g_3]\\
&\quad - [g_2]^{\otimes2}\otimes[g_3]\otimes[g_1g_3] - [g_3]\otimes[g_1g_2]\otimes[g_2g_3] + [g_1]\otimes[g_3]^{\otimes2}
\\
&=\Phi(Q_{\alpha g_1g_3, \beta g_1g_3}).
\end{align*}

On the other hand we have
\begin{align*}
[\gamma g_1g_3] - [\delta g_1g_3] &= [(g_2g_1^{-1}g_3)\cdot(g_1g_3^{-1}g_1g_3)] - [(g_1^{-1}g_3)\cdot(g_1g_3^{-1}g_2g_1g_3)]\\
&= \left([g_2g_1^{-1}g_3]\otimes[g_1g_3^{-1}g_1g_3] - [g_1g_2^{-1}g_1g_3^{-1}g_1]\right)\\
&\quad - \left([g_1^{-1}g_3]\otimes[g_1g_3^{-1}g_2g_1g_3] - [g_1g_1g_3^{-1}g_2g_1]\right).
\end{align*}
Here we have
\begin{align*}
[(g_1g_3^{-1})\cdot(g_1g_3)] &= [g_1g_3^{-1}]\otimes[g_1g_3] - [g_3^2]\\
[(g_1g_2^{-1})\cdot(g_1g_3^{-1}g_1)] &= [g_1g_2^{-1}]\otimes[g_1\cdot(g_3^{-1}g_1)] - [g_2)\cdot g_3^{-1}\cdot(g_1]\\
&= [g_1g_2^{-1}]\otimes([g_1]\otimes[g_3^{-1}g_1] - [g_3]) - ([g_3]\otimes [g_1g_2]-[g_1g_2g_3] )
\\
[(g_1g_3^{-1}g_2)\cdot(g_1g_3)] &= [g_1g_3^{-1}g_2]\otimes[g_1g_3] - [g_3^{-1}\cdot(g_2g_3^{-1})]\\
&= [g_1g_3^{-1}g_2]\otimes[g_1g_3] - [g_3^{-1}]\otimes[g_2g_3^{-1}] + [g_2],\\
[g_1\cdot(g_1g_3^{-1}g_2g_1)] &= [g_1]\otimes[g_1\cdot(g_3^{-1}g_2g_1)] - [g_3^{-1}g_2g_1]\\
&= [g_1]\otimes\left([g_1]\otimes[g_3^{-1}g_2g_1] - [g_3^{-1}g_2]\right) - [g_3^{-1}g_2g_1].
\end{align*}
Using the above identities and Lemma \ref{cyc} (a), (b), (c3), by a straight calculation we have
\begin{align*}
[\gamma g_1g_3] - [\delta g_1g_3] &= [g_1]^{\otimes2}\otimes[g_1g_2g_3] - [g_3]^{\otimes2}\otimes[g_1g_2g_3] - [g_1]^{\otimes3}\otimes[g_2g_3]\\
&\quad + [g_3]^{\otimes3}\otimes[g_1g_2] - [g_1]\otimes[g_1g_2]\otimes[g_1g_3] + [g_3]\otimes[g_1g_3]\otimes[g_2g_3]\\
&\quad + 2[g_1]\otimes[g_2g_3] - 2[g_3]\otimes[g_1g_2] - [g_1]\otimes[g_1g_3]^{\otimes2}\otimes[g_2g_3]\\
&\quad + [g_3]\otimes[g_1g_2]\otimes[g_1g_3]^{\otimes2}  + [g_1]^{\otimes2}\otimes[g_3]\otimes[g_1g_2]\\
&\quad - [g_1]\otimes[g_3]^{\otimes2}\otimes[g_2g_3] + [g_1]^{\otimes2}\otimes[g_3]\otimes[g_1g_3]\otimes[g_2g_3]\\
&\quad - [g_1]\otimes[g_3]^{\otimes2}\otimes[g_1g_2]\otimes[g_1g_3]
\\
&=\Phi(Q_{\gamma g_1g_3, \delta g_1g_3}).
\end{align*}

\noindent\fbox{$Q_{\alpha g_2g_3, \beta g_2g_3}$ and $Q_{\gamma g_2g_3, \delta g_2g_3}$}\\

We have
\begin{align*}
[\alpha g_2g_3] - [\beta g_2g_3] &=  [g_3\cdot(g_2^{-1}g_1g_2g_1^{-1}g_2g_3)] - [(g_2^{-1}g_1g_2g_1^{-1}g_3)\cdot(g_2g_3)]\\
&= \left([g_3]\otimes[g_2^{-1}g_1g_2)\cdot(g_1^{-1}g_2)\cdot (g_3] - [(g_2^{-1}g_1g_2)\cdot(g_1^{-1}g_2)]\right)\\
&\quad -\left([g_2^{-1}g_1g_2g_1^{-1}g_3]\otimes[g_2g_3] - [g_2^{-1}\cdot(g_1g_2g_1^{-1}g_2^{-1})]\right)\\
&= \Big([g_3]\otimes\left([g_3g_2^{-1}g_1g_2]\otimes[g_1^{-1}g_2] - [g_3g_2^{-1}g_1g_1]\right) - \left([g_1]\otimes[g_1^{-1}g_2] - [g_2^{-1}g_1g_1]\right)\Big) \\
&\quad - \Big([g_2^{-1}g_1g_2g_1^{-1}g_3]\otimes[g_2g_3]- \left([g_2^{-1}]\otimes[g_1g_2g_1^{-1}g_2^{-1}]-[g_2]\right)\Big).
\end{align*}
Here we have
\begin{align*}
[g_3\cdot g_2^{-1}\cdot(g_1g_2)]
&= [g_3]\otimes[g_1] + [g_2^{-1}]\otimes[g_3g_1g_2] + [g_1g_2]\otimes[g_3g_2^{-1}]\\
&\quad - [g_3]\otimes[g_2^{-1}]\otimes[g_1g_2] - [g_3g_1],\\
[(g_3g_2^{-1}g_1)\cdot g_1] 
&= [g_3g_2^{-1}g_1]\otimes[g_1] - [g_3g_2^{-1}],\\
[(g_2^{-1}g_1)\cdot g_1] &= [g_2^{-1}g_1]\otimes[g_1] - [g_2^{-1}],\\
[(g_2^{-1}g_1)\cdot g_2\cdot(g_1^{-1}g_3)] &= [g_2^{-1}g_1]\otimes[g_2g_1^{-1}g_3] + [g_2]\otimes[g_2^{-1}g_3] + [g_1^{-1}g_3]\otimes[g_1]\\
&\quad - [g_2^{-1}g_1]\otimes[g_2]\otimes[g_1^{-1}g_3] - [g_3],\\
[g_1\cdot g_2\cdot(g_1^{-1}g_2^{-1})] &= [g_1]\otimes[g_1^{-1}] + [g_2]\otimes[g_2^{-1}] + [g_1^{-1}g_2^{-1}]\otimes[g_1g_2] - [g_1]\otimes[g_2]\otimes[g_1^{-1}g_2^{-1}] - [e]
\\
 &= [g_1]^{\otimes2} + [g_2]^{\otimes2} + [g_1g_2]^{\otimes2} - [g_1]\otimes[g_2]\otimes[g_1g_2] - 2.
\end{align*}
Using the above identities and Lemma \ref{cyc} (b), (c2), by a straight calculation we have
\begin{align*}
[\alpha g_2g_3] - [\beta g_2g_3] &= [g_2]^{\otimes3} + [g_2]\otimes[g_3]^{\otimes2} + [g_2]\otimes[g_2g_3]^{\otimes2} - [g_2]^{\otimes2}\otimes[g_3]\otimes[g_2g_3]\\
&\quad -4[g_2] +[g_1]\otimes[g_2]^{\otimes2}\otimes[g_3]\otimes[g_1g_2g_3] - [g_1]\otimes[g_2]\otimes[g_2g_3]\otimes[g_1g_2g_3]\\
&\quad - [g_1]\otimes[g_2]\otimes[g_3]\otimes[g_1g_3] - [g_2]\otimes[g_3]\otimes[g_1g_2]\otimes[g_1g_2g_3]\\
&\quad + [g_1g_2]\otimes[g_2g_3]\otimes[g_1g_2g_3] - [g_1]\otimes[g_2]^{\otimes2}\otimes[g_1g_2] + [g_2]\otimes[g_1g_2]^{\otimes2}\\
&\quad + [g_3]\otimes[g_1g_2]\otimes[g_1g_3] - [g_1]\otimes[g_3]\otimes[g_1g_2g_3]\\
&\quad + [g_1]\otimes[g_1g_3]\otimes[g_2g_3] + [g_1]^{\otimes2}\otimes[g_2]
\\
&=\Phi(Q_{\alpha g_2g_3,  \beta g_2g_3}).
\end{align*}

On the other hand, we have
\begin{align*}
[\gamma g_2g_3] - [\delta g_2g_3] &= [(g_2g_1^{-1}g_3g_1g_3^{-1})\cdot(g_2g_3)] - [g_1^{-1}g_3g_1g_3^{-1})\cdot g_2\cdot(g_2g_3]\\
&= \left([g_2g_1^{-1}g_3g_1g_3^{-1}]\otimes[g_2g_3] - [(g_1^{-1}g_3g_1g_3^{-1})\cdot g_3^{-1}]\right)\\
&\quad - \left([g_2g_3)\cdot (g_1^{-1}g_3)\cdot(g_1g_3^{-1}]\otimes[g_2] - [g_3g_1^{-1}g_3g_1g_3^{-1}]\right)\\
&= [g_2g_1^{-1}g_3g_1g_3^{-1}]\otimes[g_2g_3] - [g_1^{-1}g_3g_1g_3^{-1}]\otimes[g_3^{-1}] + [g_1^{-1}g_3g_1]\\
&\quad -\left([g_1^{-1}g_3]\otimes[g_1g_3^{-1}g_2g_3] - [g_1g_1g_3^{-1}g_2]\right)\otimes[g_2] + [g_3].
\end{align*}
Here we  have
\begin{align*}
[(g_2g_1^{-1})\cdot g_3\cdot(g_1g_3^{-1})] &= [g_2g_1^{-1}]\otimes[g_1] + [g_3]\otimes[g_2g_3^{-1}]+ [g_1g_3^{-1}]\otimes[g_2g_1^{-1}g_3]\\
&\quad  - [g_2g_1^{-1}]\otimes[g_3]\otimes[g_1g_3^{-1}] - [g_2],\\
[g_1^{-1}\cdot(g_3g_1)\cdot g_3^{-1}] &= [g_1^{-1}]\otimes [g_1] + [g_1g_3]\otimes [g_1^{-1}g_3^{-1}] + [g_3^{-1}]\otimes [g_3] - [g_1^{-1}]\otimes[g_1g_3]\otimes[g_3^{-1}] -[e]\\
&= [g_1]^{\otimes2} + [g_1g_3]^{\otimes2} + [g_3]^{\otimes2} - [g_1]\otimes[g_1g_3]\otimes[g_3] -2,
\\
[g_1g_3^{-1}g_2g_3] &= [(g_3g_1)\cdot g_3^{-1}\cdot g_2]\\
&= [g_3g_1]\otimes[g_3^{-1}g_2] + [g_3^{-1}]\otimes[g_3g_1g_2] + [g_2]\otimes[g_1]\\
&\quad - [g_3g_1]\otimes[g_3^{-1}]\otimes[g_2] - [g_1g_2],\\
[g_1\cdot(g_1g_3^{-1}g_2)] &= [g_1]\otimes[g_1g_3^{-1}g_2] - [g_3^{-1}g_2].
\end{align*}
Using the above identities and Lemma \ref{cyc} (b) (c3), by a straight calculation we have
\begin{align*}
[\gamma g_2g_3] - [\delta g_2g_3] &= - [g_3]^{\otimes3} - [g_1]^{\otimes2}\otimes[g_3] - [g_3]\otimes[g_1g_3]^{\otimes2} + [g_1]\otimes[g_3]^{\otimes2}\otimes[g_1g_3]\\
&\quad + 4[g_3] - [g_1]\otimes[g_2]\otimes[g_3]^{\otimes2}\otimes[g_1g_2g_3] + [g_2]\otimes[g_3]\otimes[g_1g_3]\otimes[g_1g_2g_3]\\
&\quad + [g_1]\otimes[g_2]\otimes[g_3]\otimes[g_1g_2] + [g_1]\otimes[g_3]\otimes[g_2g_3]\otimes[g_1g_2g_3]\\
&\quad - [g_1g_3]\otimes[g_2g_3]\otimes[g_1g_2g_3] + [g_2]\otimes[g_3]^{\otimes2}\otimes[g_2g_3]  - [g_3]\otimes[g_2g_3]^{\otimes2}\\
&\quad - [g_1]\otimes[g_1g_2]\otimes[g_2g_3] +[g_1]\otimes[g_2]\otimes[g_1g_2g_3] - [g_2]\otimes[g_1g_2]\otimes[g_1g_3]\\
&\quad - [g_2]^{\otimes2}\otimes[g_3]
\\
&=\Phi(Q_{\gamma g_2g_3, \delta g_2g_3}).
\end{align*}

\noindent\fbox{$Q_{\alpha g_1g_2g_3, \beta g_1g_2g_3}$ and $Q_{\gamma g_1g_2g_3, \delta g_1g_2g_3}$}\\

We have
\begin{align*}
[\alpha g_1g_2g_3] - [\beta g_1g_2g_3] &= [g_3\cdot (g_2^{-1}g_1g_2g_1^{-1}g_1g_2g_3)] - [(g_2^{-1}g_1g_2g_1^{-1}g_3)\cdot (g_1g_2g_3)]\\
&= \left([g_3]\otimes[g_2^{-1}g_1)\cdot g_2\cdot (g_2g_3 ] - [g_1g_2]\right)\\
&\quad - \left([g_2^{-1}g_1g_2g_1^{-1}g_3]\otimes[g_1g_2g_3] - [(g_2^{-1}g_1g_2g_1^{-1})\cdot(g_2^{-1}g_1^{-1})]\right)\\
&= \Big([g_3]\otimes([g_2g_3g_2^{-1}g_1]\otimes[g_2] - [g_3g_2^{-1}g_1]) - [g_1g_2]\Big)\\
&\quad - \Big( [g_2^{-1}g_1g_2g_1^{-1}g_3]\otimes[g_1g_2g_3] - [g_2^{-1}g_1g_2g_1^{-1}]\otimes[g_2^{-1}g_1^{-1}] + [g_1g_2]\Big).
\end{align*}
Here we have
\begin{align*}
[g_2g_3g_2^{-1}g_1] &=[g_2^{-1} \cdot g_1 \cdot(g_2g_3)] \\
&= [g_2^{-1}]\otimes[g_1g_2g_3] + [g_1]\otimes[g_3] + [g_2g_3]\otimes[g_2^{-1}g_1]\\
&\quad - [g_2^{-1}]\otimes[g_1]\otimes[g_2g_3] - [g_1g_3],\\
[(g_2^{-1}g_1)\cdot g_2\cdot (g_1^{-1}g_3)] &=  [g_2^{-1}g_1]\otimes[g_2g_1^{-1}g_3] + [g_2]\otimes[g_2^{-1}g_3] + [g_1^{-1}g_3]\otimes[g_1]\\
&\quad - [g_2^{-1}g_1]\otimes[g_2]\otimes[g_1^{-1}g_3] - [g_3],\\
[g_2^{-1}\cdot(g_1g_2)\cdot g_1^{-1}] &=[g_2^{-1}]\otimes [g_2] + [g_1g_2]\otimes [g_2^{-1}g_1^{-1}] + [g_1^{-1}]\otimes [g_1] - [g_2^{-1}]\otimes[g_1g_2]\otimes[g_1^{-1}] - [e]
\\
&= [g_2]^{\otimes2} + [g_1g_2]^{\otimes2} + [g_1]^{\otimes2} - [g_2]\otimes[g_1g_2]\otimes[g_1] - 2.
\end{align*}
Using the above identities and Lemma \ref{cyc} (b), (c1) and (c2), by a straight calculation we have
\begin{align*}
[\alpha g_1g_2g_3] - [\beta g_1g_2g_3] &= - [g_2]\otimes[g_3]\otimes[g_1g_2]\otimes[g_2g_3] + [g_1g_2]\otimes[g_3]^{\otimes2} + [g_1]\otimes[g_3]\otimes[g_2g_3]\\
&\quad - [g_2]\otimes[g_3]\otimes[g_1g_3] - 4 [g_1g_2] + [g_2]^{\otimes2}\otimes[g_1g_2] + [g_1g_2]^{\otimes3}\\
&\quad + [g_1]^{\otimes2}\otimes[g_1g_2] - [g_1]\otimes[g_2]\otimes[g_1g_2]^{\otimes2} - [g_1]\otimes[g_2]\otimes[g_1g_2g_3]^{\otimes2}\\
&\quad + [g_1g_2]\otimes[g_1g_2g_3]^{\otimes2} + [g_1]\otimes[g_2]\otimes[g_3]\otimes[g_1g_2]\otimes[g_1g_2g_3]\\
&\quad - [g_3]\otimes[g_1g_2]^{\otimes2}\otimes[g_1g_2g_3] + [g_2]\otimes[g_2g_3]\otimes[g_1g_2g_3]\\
&\quad - [g_1]^{\otimes2}\otimes[g_3]\otimes[g_1g_2g_3] + [g_1]\otimes[g_1g_3]\otimes[g_1g_2g_3]
\\
&=\Phi(Q_{\alpha g_1g_2g_3, \beta g_1g_2g_3}).
\end{align*}

On the other hand we have
\begin{align*}
[\gamma g_1g_2g_3] - [\delta g_1g_2g_3] &= [(g_2g_1^{-1}g_3)\cdot(g_1g_3^{-1}g_1g_2g_3)] - [(g_1^{-1}g_3)\cdot(g_1g_3^{-1}g_2g_1g_2g_3)]\\
&= \left([g_2g_1^{-1}g_3]\otimes[g_1g_3^{-1}g_1g_2g_3] - [(g_1g_2^{-1}g_1g_3^{-1})\cdot(g_1g_2)]\right)\\
&\quad - \left([g_1^{-1}g_3]\otimes[(g_1g_3^{-1}g_2)\cdot(g_1g_2g_3)] - [g_1\cdot(g_1g_3^{-1}g_2g_1g_2)]\right)\\
&= \left([g_2g_1^{-1}g_3]\otimes[g_1g_3^{-1}g_1g_2g_3] - [g_1g_2^{-1}g_1g_3^{-1}]\otimes[g_1g_2]+ [g_2^{-1}g_1g_3^{-1}g_2^{-1}]\right) \\
&\quad - \Big([g_1^{-1}g_3]\otimes\left([g_1g_3^{-1}g_2]\otimes[g_1g_2g_3] - [g_2^{-1}g_3g_2g_3]\right)\\
&\quad + [g_1]\otimes[g_1g_3^{-1}g_2g_1g_2] - [g_3^{-1}g_2g_1g_2]\Big).
\end{align*}
Here we have
\begin{align*}
[(g_1g_3^{-1})\cdot(g_1g_2g_3)] &= [g_1g_3^{-1}]\otimes[g_1g_2g_3] - [g_3\cdot(g_2g_3)]\\
&= [g_1g_3^{-1}]\otimes[g_1g_2g_3] - [g_3]\otimes[g_2g_3] + [g_2],\\
[(g_1g_2^{-1})\cdot(g_1g_3^{-1})] &= [g_1g_2^{-1}]\otimes[g_1g_3^{-1}] - [g_2g_3^{-1}],\\
[g_2^{-1}\cdot(g_1g_3^{-1}g_2^{-1})] &= [g_2^{-1}]\otimes[(g_1g_3^{-1})\cdot g_2^{-1}] - [g_1g_3^{-1}]\\
&= [g_2^{-1}]\otimes([g_1g_3^{-1}]\otimes[g_2^{-1}] - [g_1g_3^{-1}g_2]) - [g_1g_3^{-1}],\\
[(g_2^{-1}g_3)\cdot(g_2g_3)] &= [g_2^{-1}g_3]\otimes[g_2g_3] - [g_2^{2}]\\
&= [g_2^{-1}g_3]\otimes[g_2g_3] - [g_2]^{\otimes2} + 2,\\
[(g_1g_3^{-1}g_2)\cdot(g_1g_2)] &= [g_1g_3^{-1}g_2]\otimes[g_1g_2] - [g_3],\\
[(g_3^{-1}g_2)\cdot(g_1g_2)] &= [g_3^{-1}g_2]\otimes[g_1g_2] - [g_3g_1].
\end{align*}
Using the above identities and Lemma \ref{cyc} (b), (c1) and (c3), by a straight calculation we have
\begin{align*}
[\gamma g_1g_2g_3] - [\delta g_1g_2g_3] &= -  [g_1]\otimes[g_1g_3]\otimes[g_2g_3]\otimes[g_1g_2g_3] + [g_3]\otimes[g_1g_2]\otimes[g_1g_3]\otimes[g_1g_2g_3]\\
&\quad + [g_1]\otimes[g_1g_2]\otimes[g_1g_2g_3] - [g_3]\otimes[g_2g_3]\otimes[g_1g_2g_3]\\
&\quad + [g_3]^{\otimes2}\otimes[g_1g_2]\otimes[g_2g_3] - [g_1]^{\otimes2}\otimes[g_1g_2]\otimes[g_2g_3] + [g_1g_3]\otimes[g_2g_3]^{\otimes2}\\
&\quad - [g_1g_2]^{\otimes2}\otimes[g_1g_3] + [g_1]^{\otimes2}\otimes[g_3]\otimes[g_2g_3]\otimes[g_1g_2g_3]\\
&\quad - [g_1]\otimes[g_3]^{\otimes2}\otimes[g_1g_2]\otimes[g_1g_2g_3] + [g_1]\otimes[g_3]\otimes[g_1g_2]^{\otimes2}\\
&\quad - [g_1]\otimes[g_3]\otimes[g_2g_3]^{\otimes2} + [g_1]\otimes[g_2]\otimes[g_2g_3] - [g_2]\otimes[g_3]\otimes[g_1g_2]
\\
&=\Phi(Q_{\gamma g_1g_2g_3, \delta g_1g_2g_3}).
\end{align*}

Hence we have the assertion.
\end{proof}

\end{document}